\documentclass[12pt,twoside,leqno]{amsart}
\usepackage{amsmath}
\usepackage{amssymb}
\usepackage{amsxtra}
\usepackage{amscd}
\usepackage{verbatim}
\usepackage{amsthm}
\usepackage[mathscr]{eucal}
\usepackage{amsmath,amssymb,latexsym,amsfonts,graphicx}
\usepackage{array}
\usepackage{url}
\usepackage[all]{xy}
\usepackage{supertabular}

\theoremstyle{plain}
\newtheorem{thm}[subsection]{Theorem}
\newtheorem{prop}[subsection]{Proposition}

\newtheorem{lem}[subsection]{Lemma}
\newtheorem{conj}[subsection]{Conjecture}

\theoremstyle{definition}

\newtheorem{eg}[subsection]{Example}

\input cyracc.def 
\font\tencyr=wncyr10
\def\sha{\text{\tencyr\cyracc{Sh}}}
\font\eightcyr=wncyr8
\font\tabfont=cmr6
\def\smallsha{\text{\eightcyr\cyracc{Sh}}}

\newcommand{\vabove}[2]{\genfrac{}{}{0pt}{}{#1}{#2}}
\def\sm#1#2#3#4{\genfrac{(}{.}{0pt}{1}{#1}{#3}\genfrac{.}{)}{0pt}{1}{#2}{#4}}
\let\iso\cong
\let\tensor\otimes
\newcommand{\sgn}{{\rm sgn}}
\newcommand{\Frob}{{\rm Frob}}

\newcommand{\Q}{\mathbb Q}
\newcommand{\qp}{{\mathbb Q}_p}
\newcommand{\R}{\mathbb R}
\newcommand{\C}{\mathbb C}
\newcommand{\Z}{\mathbb Z}
\newcommand{\La}{\Lambda}
\newcommand{\cyc}{{\rm cyc}}
\newcommand{\Gal}{{\rm Gal}}
\newcommand{\Aut}{{\rm Aut}}
\newcommand{\ord}{{\rm ord}}
\newcommand{\can}{{\rm can}}
\renewcommand{\det}{{\rm det}}
\newcommand{\sel}{{\rm Sel}}
\renewcommand{\ker}{{\rm Ker}}
\newcommand{\Hom}{{\rm Hom}}
\newcommand{\cK}{\mathcal K}
\newcommand{\cL}{\mathcal L}
\newcommand{\cF}{\mathcal F}
\newcommand{\cO}{\mathcal O}
\newcommand{\ap}{{A_{p^{\infty}}}}
\newcommand{\apo}{{A^0_{p^{\infty}}}}
\newcommand{\fc}{{\cF}^{\cyc}}
\newcommand{\qc}{{\q^{\cyc}}}
\newcommand{\kc}{{\cK^{\cyc}}}
\newcommand{\clp}{{\cL}_p^+}
\newcommand{\clm}{{\cL}_p^-}
\newcommand{\f}{\frak f}
\newcommand{\g}{\frak g}
\newcommand{\n}{\frak n}
\renewcommand{\a}{\frak a}
\newcommand{\Fin}{F_{\infty}}

\begin{document}

\title{Non-commutative Iwasawa theory for modular forms}
\author{J. Coates, T. Dokchitser, Z. Liang, W. Stein, R. Sujatha}
\address{J. Coates, Centre for Mathematical Sciences, Wilberforce Road, Cambridge CB3 0WB, England}
\address{T. Dokchitser, Department of Mathematics, University Walk, Bristol BS8 1TW, United Kingdom}
\address{Z. Liang, School of Mathematical Sciences, Capital Normal University, Xisanhuanbeilu 105, Haidan District, Beijing, China}
\address{W. Stein, Department of Mathematics, University of Washington, Seattle, Box 354350 WA 98195, USA}
\address{R. Sujatha, Mathematics Department, 1984, Mathematics Road, University of British Columbia, Vancouver BC, Canada V6T 1Z2}
\email{J.H.Coates@dpmms.cam.ac.uk, tim.dokchitser@bristol.ac.uk,
liangzhb@gmail.com, wstein@uw.edu, sujatha@math.ubc.ca}
\thanks{%
T. Dokchitser was supported by a Royal Society University Research Fellowship,
Z. Liang by the National Natural Sciences Foundation of China 
  (Grant Nos. 11001183 and 11171231),
W. Stein from NSF Grant DMS-1015114, and R. Sujatha from NSERC Grant 22R65339.}

\subjclass[2000]{Primary 11F67; Secondary 11F33, 11R23}

\begin{abstract}
The aim of the present paper is to give evidence, largely numerical, in 
support of the non-commutative main conjecture of Iwasawa theory for the 
motive of a primitive modular form of weight $k>2$ over the Galois 
extension of $\Q$ obtained by adjoining to $\Q$ all $p$-power roots of unity, and 
all $p$-power roots of a fixed integer $m > 1$. The predictions of the main 
conjecture are rather intricate in this case because there is more than 
one critical point, and also there is no canonical choice of periods. 
Nevertheless, our numerical data agrees perfectly with all aspects of the 
main conjecture, including Kato's mysterious congruence between the 
cyclotomic Manin $p$-adic $L$-function, and the cyclotomic $p$-adic 
$L$-function of a twist of the motive by a certain non-abelian 
Artin character of the Galois group of this extension.
\end{abstract}

\maketitle

\section{Introduction}
Let $z$ be a variable in the upper half complex plane, and put $q = e^{2{\pi}iz}$. Let
\begin{equation}\label{f}
f(z)=\overset{\infty}{\underset{n=1}{\sum}}\, a_nq^n ,
\end{equation}
be a primitive cusp form of conductor $N$ (in the sense of \cite{MI}),
with trivial character, and weight $k >2.$ For simplicity, we shall always assume that the Fourier coefficients
$a_n ~(n\geq 1)$ of $f$ are in $\Q$.  Let $p$ be an odd prime number. The aim of the present paper is to provide some
 evidence, largely numerical, for the validity of the non-commutative main conjecture of Iwasawa theory for the motive $M(f)$ attached to $f$ over the $p$-adic Lie extension
$$
\Fin=\Q(\mu_{p^{\infty}}, m^{1/p^n}, n=1,2,\dots),
$$
which is obtained by adjoining to $\Q$ the group $\mu_{p^{\infty}}$ of all $p$-power roots of unity, and all $p$-power roots of some fixed integer $m >1.$   In this case, the analytic continuation and functional equation for the complex $L$-function $L(f,\phi,s)$ of $f$ twisted by any Artin character $\phi$ of the Galois group of $\Fin$ over $\Q$ are well-known consequences  of the theory of automorphic base change.  The points $s=1,\dots, k-1$ are critical for all of the complex $L$-functions $L(f,\phi,s)$, and we show that essentially the same arguments as in \cite{BD1} enable one to prove the expected algebraicity statement at these points. Moreover, these values are all non-zero, except perhaps for the central value $s=k/2$; in particular, there is always at least one non-zero critical value since $k>2.$

\medskip

In \cite{CFKSV}, a precise main conjecture was formulated for an elliptic curve over any $p$-adic Lie extension of a number field $F$ containing the cyclotomic $\Z_p$-extension of $F$, and under the assumption that the elliptic curve is ordinary at  the prime $p$. This was generalized to arbitrary ordinary motives in \cite{FK}, and it is a special case of the main conjecture of
\cite{FK} which we consider here.  Thus we assume that $p$ is an odd prime number such that $(p,a_p)=(p,N)=1$. One of the underlying ideas of the non-commutative main conjecture  is to prove the existence of a $p$-adic $L$-function, which interpolates a canonical normalization of the critical values $L(f,\phi,n),$ where $n=1,\dots, k-1,$ and $\phi$ runs over all Artin representations of the Galois group
$$
G = \Gal(\Fin/\Q).
$$
We denote these normalized $L$-values by ${\cL}_p^{\can}(f,\phi,n)$ (for the precise definition, see formulae \eqref{56}, \eqref{56''} and \eqref{57'} in \S 5). The definition of these normalized $L$-values requires making a choice of canonical periods for the form $f$, and, until  such a time as the main conjectures of non-commutative Iwasawa theory are fully proven,  we are only able to make an educated guess at present as to what these canonical periods should be. However, as we explain in \S5, Manin's work on the construction of the $p$-adic $L$-function for our modular form $f$ over the field $\Q(\mu_{p^{\infty}})$ gives some information about these canonical periods, which is relevant for our numerical examples.

\medskip

As we explain in more detail in \S 5, the existence of a $p$-adic $L$-function attached to $f$ over the non-abelian extension $\Fin$ of $\Q$, when combined with the work of Kato \cite{K1}, implies the existence of the following mysterious congruence between two $p$-adic $L$-functions attached to $f$ over certain abelian sub-extensions of $\Fin/\Q$. We are very grateful to M. Kakde for explaining to us how this congruence follows from Kato's work. Let $\sigma$ denote the $(p-1)$-dimensional representation of $G$ given by the direct sum of the irreducible representations of $\Gal(\Q(\mu_p)/\Q)$. Let $\rho$ be the unique irreducible representation of dimension $p-1$ of the Galois group of the field
$$
 F=\Q(\mu_p,m^{1/p})
 $$
over $\Q$, where we now assume that $m>1$ is $p$-power free. Write $\Q^{\cyc}$ for the cyclotomic $\Z_p$-extension of $\Q$, and $\Xi$ for the group of irreducible characters of finite order of $\Gamma =
\Gal(\Q^{\cyc}/\Q)$. Further, let $\chi_p$ denote the character giving the action of $\Gal(\bar{\Q}/{\Q})$  on $\mu_{p^{\infty}}.$ We fix a topological generator $\gamma$ of $\Gamma$, and put $u=\chi_p(\gamma).$ The work of Manin \cite{M} proves that there exists  a unique power series $H(\sigma, T)$ in the ring $R=\Z_p[[T]]$ such that
\begin{equation}\label{1.}
H(\sigma,\psi(\gamma)u^r-1)={\cL}_p^{\can}(f,\sigma\psi,k/2+r),
\end{equation}
for all $\psi$ in $\Xi$, and all integers $r$ with $-k/2+1\leq r \leq k/2-1.$
On the other hand, the conjectural existence of a good $p$-adic $L$-function for $f$ over the field $\Fin$ would imply, in particular, the existence of a power series $H(\rho,T)$ in the ring $R$ such that
\begin{equation}\label{1.1}
H(\rho,\psi(\gamma)u^r-1)={\cL}_p^{\can}(f,\rho\psi,k/2+r),
\end{equation}
for all $\psi$ in $\Xi$, and all integers $r$ with $-k/2+1\leq r \leq k/2-1.$
Then Kato's work \cite{K1} implies  the following conjectural congruence between formal power series
\begin{equation}\label{1.11}
H(\rho,T) \equiv H(\sigma,T)\!\!\!\mod\, pR.
\end{equation}
This conjectural congruence in $R$ has the following consequences for our critical $L$-values. Firstly, on evaluation of our power series at the relevant point in $p\Z_p$, we deduce from \eqref{1.} and \eqref{1.1} that the congruence
\begin{equation}\label{1'}
{\cL}_p^{\can}(f,\rho,n) \equiv {\cL}_p^{\can}(f,\sigma,n)\,{\rm mod}\, p\Z_p
\end{equation}
should hold for $n=1,\dots, k-1.$ Secondly, if we assume the additional property that
\begin{equation}\label{1.3}
L(f,\sigma,k/2)=L(f,\rho,k/2)=0,
\end{equation}
then we would have that $H(\rho,T)$ and $H(\sigma,T)$ both belong to the ideal $TR.$ It is then clear from \eqref{1.}, \eqref{1.1} and \eqref{1.11} that the stronger congruence
\begin{equation}\label{1''}
{\cL}_p^{\can}(f,\rho,n) \equiv {\cL}_p^{\can}(f,\sigma,n)\,{\rm mod}\, p^2\Z_p
\end{equation}
should hold for $n=1,\dots, k-1.$

\medskip

Our numerical computations (see \S 6) verify the first congruence \eqref{1'} for the prime $p=3$ and a substantial range of cube free integers $m>1,$ for three forms $f$ of weight 4 and conductors 5, 7, 121, and one form $f$ of weight 6 and conductor 5, all of which are ordinary at 3. These computations require us to determine numerically the Fourier coefficients $a_n$ of these forms $f$ for $n$ in the range $1 \leq n \leq 10^8$.  In addition, for the two forms of weight 4 and conductors 7 and 121, we prove that \eqref{1.3}  holds  for all integers $m>1$, and happily, our numerical results show that the sharper congruence \eqref{1''} holds for these two forms and the prime $p=3$ for a good range of cube free integers $m>1.$ When $f$ is a complex mutliplication form, some cases of the congruence \eqref{1.11} have already been established theoretically by Delbourgo and Ward \cite{DT} and Kim \cite{DK}. However, when $f$ is not a complex multiplication form, our numerical data seems to provide the first hard evidence in support of the mysterious non-abelian congruence \eqref{1.11} between abelian $p$-adic $L$-functions.

\medskip

We warmly thank T. Bouganis, M. Kakde, and D. Kim
for very helpful advice on the writing of this paper.

\section{Algebraicity of $L$-values}

As in the Introduction, let $f$ given by \eqref{f}
be a primitive cusp form of conductor $N\geq 1$ with trivial character and weight $k >2$
(thus $k$ is necessarily even).
 For simplicity, we always assume that the Fourier coefficients $a_n\,(n \geq 1)$ of $f$ belong to $\Q.$
The complex $L$-function attached to $f$ is
\begin{equation}\label{2}
L(f,s)=\sum_{n=1}^{\infty}\, a_n/n^s.
\end{equation}
This $L$-function has the following Euler product.  For any prime $p$,  let
\begin{equation}\label{7}
\tau_p\,:\, \Gal (\bar {\Q}/\Q) \to \Aut_{\Q_p}(V_p)
\end{equation}
be the $p$-adic Galois representation attached to $f$; here $V_p$ is a two dimensional vector space over the field $\Q_p$ of $p$-adic numbers. If $q$ is any prime distinct from $p$, define the polynomial
\begin{equation}\label{8}
P_q(f,X)={\rm det}(1-\tau_p({\rm Frob}_q^{-1})X\mid V_p^{I_q}),
\end{equation}
where $I_q$ is the inertial subgroup of the decomposition group of any fixed prime of $\bar{\Q}$ above $q$, and ${\rm Frob}_q$ denote the Frobenius automorphism of $q$.  Moreover,  if $(q,N)=1$,  we have
\begin{equation}\label{9}
P_q(f,X)=1-a_qX + q^{k-1} X^2.
\end{equation}
Then
\begin{equation}\label{10}
L(f,s)=\prod_q P_q(f,q^{-s})^{-1}
\end{equation}
when ${\rm Re}(s)> 1+ (k-1)/2.$ Defining
\begin{equation}\label{3}
\La(f,s)=N^{s/2}(2\pi)^{-s}\Gamma(s)L(f,s),
\end{equation}
we know, since Hecke, that $\La(f,s)$ is entire and satisfies the functional equation
\begin{equation}\label{4}
\La(f,s)=w(f)\La(f,k-s)
\end{equation}
where $w(f)=\pm 1$ is the sign in the functional equation. The critical values of $L(f,s)$ are at the points
$s=1,\dots,k-1.$  Following Shimura \cite{S1}, \cite{S2}, we introduce the following  naive periods for $f$, which we have normalized in view of our later numerical calculations.
Define
\begin{equation}\label{5}
\Omega_-(f)=\,i w(f) (2\pi )^{-1}L(f,1).
\end{equation}
Since the Euler product for $L(f,s)$ converges to a positive real number when $s$ is real and $s> 1+(k-1)/2,$ it is clear from the functional equation
\eqref{4}, that $\Omega_-(f)$ is purely imaginary in the upper half plane.  Motivated again by numerical calculations, we assume throughout
the following simplifying hypothesis (see \cite{DSW} for examples in which
this hypothesis fails).

\medskip

\noindent{\bf Hypothesis H1:}   $L(f,2)\neq 0$
when $k=4.$

\medskip

\noindent We then define
\begin{equation}\label{6}
\Omega_+(f)=w(f)(2\pi)^{-2} \,L(f,2).
\end{equation}
Again, $\Omega_+(f)$ is always a positive real number when $k>4$, and presumably (it would, of course, be implied by the generalized Riemann Hypothesis) this remains true even when $k=4$, although this value is outside the region of convergence of the Euler product.

\begin{thm}{\rm (See \cite{S1},\cite{S2})}\label{shim1}
(i) If $n$ is an odd integer such that $1\leq n \leq k-1$,  then
$$
(2\pi i)^{-n}L(f,n)/\Omega_-(f) \in \Q;
$$
(ii) If $n$ is an even integer such that  $ 1 \leq n \leq k-1$, then
$$
(2\pi i)^{-n}L(f,n)/\Omega_+(f) \in \Q.
$$
\end{thm}

\medskip

In what follows, we shall mainly be interested in the $L$-functions of $f$ twisted by Artin characters. We rapidly recall the definitions of these $L$-functions. By an Artin representation, we mean a homomorphism
\begin{equation}\label{10'}
\phi\,:\, \Gal(\bar{\Q}/\Q) \to \Aut_{\bar {\Q}}(W)
\end{equation}
which factors through the Galois group of a finite  extension of $\Q$; here $W$ is a vector space of finite dimension over $\bar{\Q}$. Put
$$
d(\phi)=\,{\rm dim}_{\bar{\Q}}(W).
$$
For each prime $p$, let
$$
M_p(f)=V_p\otimes_{\Q_p} \bar{\Q}_p,~~~ M_p(\phi)=W\otimes_{\bar {\Q}} \bar{\Q}_p.
$$
Then
\begin{equation}\label{eulp}
L(f,\phi,s)=\underset{q}{\Pi}\,P_q(f,\phi,q^{-s})^{-1},
\end{equation}
where
\begin{equation}\label{10''}
P_q(f,\phi,X)={\rm det}\left((1-{\rm  Frob}_q^{-1}X)\mid (M_p(f)\otimes_{\bar{\Q}_p}M_p(\phi))^{I_q}\right) ~~~(q\neq p)
\end{equation}
is the Euler product attached to the tensor product Galois representation  $M_p(f)\otimes_{\bar{\Q}_p} M_p(\phi).$
The Euler product \eqref{eulp} converges in the region ${\rm Re}(s) > 1+(k-1)/2. $
It is one of the fundamental problems of number theory to prove the analytic continuation and the following
conjectural functional equation for $L(f,\phi,s).$ Let $N(f,\phi)$ be the conductor of the family of  $p$-adic representations  $M_p(f)\otimes_{\bar{\Q}_p}M_p(\phi),$ and define
\begin{equation}\label{11}
\La(f,\phi,s)=N(f,\phi)^{s/2}\left((2\pi)^{-s}\Gamma(s)\right)^{d(\phi)}L(f,\phi,s).
\end{equation}
Then conjecturally
\begin{equation}\label{12}
\La(f,\phi,s)=w(f,\phi)\Lambda(f,\hat{\phi},k-s),
\end{equation}
where $w(f,\phi)$ is an algebraic number of complex absolute value 1, and $\hat{\phi}$ is the contragredient representation of $\phi.$  There is one important case in which this result is known.
\begin{thm}\label{sha}
Let $K$ be any finite Galois extension of $\Q$ with Galois group $\Gal(K/\Q)$ abelian. Let $\psi$ be an abelian character of $K$ and define $\phi$ to be the induced character of $\Gal(\bar{\Q}/\Q).$ Then $\La(f,\phi,s)$ is entire and satisfies the functional equation \eqref{12}.
\end{thm}
\proof
Since $K$ is an abelian extension of $\Q$, the base change of $f$ to $K$, which we denote by $\pi_K(f)$, exists as a cuspidal automorphic representation of ${\rm GL}_2/K$. The results of Jacquet-Langlands  then establish the analytic continuation and functional equation for the automorphic $L$-function of $\pi_K(f)$, twisted by the abelian character $\psi$ of  $K$, which
we view as a Hecke character of ${\rm GL}_1/K.$  We denote this automorphic $L$-function by $L(\pi_K(f),\psi,s).$ On the other hand, by the theory of base change, and the local Langlands correspondence for ${\rm GL}_2$,  $L(\pi_K(f),\psi,s)$ coincides with $L(f,\phi,s)$ defined by the Euler product \eqref{eulp}.
This completes the proof on
 noting that  the functional equation \eqref{12} coincides with the automorphic functional equation.
\qed

\medskip

The following conjectural generalisation of Theorem \ref{shim1} is folklore. Given an Artin representation $\phi$ as in \eqref{10'}, define $d^+(\phi)$ (resp. $d^-(\phi)$) to be the dimension of the subspaces of $W$ on which complex conjugation acts like $+1$,  (resp. as $-1$). If $n$ is any integer, we write
\begin{equation}\label{14}
d^+_n(\phi) =  d^{(-1)^n}(\phi),~~~~d^-_n(\phi) =  d^{(-1)^{n+1}}(\phi).
\end{equation}

\begin{conj}\label{alg}
For every Artin representation $\phi$ of $\Gal({\bar \Q}/\Q),$ and all integers $n=1,\ldots,k-1$, we have
\begin{equation}\label{15}
\frac{L(f,\phi,n)}{\left((2 \pi i)^{nd(\phi)}\times\Omega_+(f)^{d^+_n(\phi)}\times\Omega_-(f)^{d^-_n(\phi)}\right)}
\in \bar\Q.
\end{equation}
\end{conj}

\medskip

Of course, when $\phi$ has dimension 1, this conjecture is a well known consequence of the theory of higher weight modular symbols. However, as in \cite{BD1}, we shall study special cases of this conjecture by using the work of Shimura \cite{S1} on the special values of Rankin products of Hilbert modular forms for totally real number fields.
Let $K$ be an arbitrary totally real field, which is Galois over $\Q$, with  $\Gal(K/\Q)$ abelian.
Take $\g$ to be any Hilbert modular form relative to $K$, which corresponds to an Artin representation
$\theta$ of dimension 2 of $\Gal(\bar{\Q}/K)$.  The form $\g$ has  parallel weight 1 and level equal to the conductor of $\theta$. We denote the Artin $L$-series of $\theta$ by
$$
L(\theta, s)=\underset{\frak a}{\sum}\, c(\frak a)(N {\frak a})^{-s},
$$
where $\frak a$ runs over all integral ideals of $K$. Further,  let $L(f/K,s)$ be the complex $L$-function attached to the restriction of the Galois representation \eqref{3} to $\Gal (\bar{\Q}/K),$ and write
$$
L(f/K,s)=\underset{\frak a}{\sum}\, b(\frak a)(N \frak a)^{-s},
$$
for its corresponding Dirichlet series. Since we have assumed $K$ to be an abelian extension
of $\Q$, the base change to $K$ of our modular form $f$ also exists as a primitive cusp form for the Hilbert modular group of $K$. We denote this base change by $\frak f_K$.  It has parallel weight $k$, trivial character, and level dividing $N{\mathcal O}_K$, where ${\mathcal O}_K$ is the ring of integers of $K$. In what follows,
we will be primarily interested in the complex $L$-series defined by the tensor product of the Artin representation $\theta$ and the Galois representation \eqref{3} of $f$ restricted to $\Gal(\bar{\Q}/K).$  We denote this $L$-series by $L(\f_K,\theta,s)$, and recall that it is defined by the Euler product
$$
L(\f_K,\theta,s)=\underset{v}{\prod}\, P_v(\f_K,\theta,(Nv)^{-s})^{-1},
$$
where $v$ runs over all finite places of $K$, and
\begin{equation}\label{19}
P_v(\f_K,\theta,X)= \det \left(1-{\rm Frob}_v^{-1}X\mid \left( M_p(f)\otimes_{\bar{\Q}_p}W_{\theta}\right)^{I_v}\right);
\end{equation}
here $W_{\theta}$ is a two dimensional $\bar{\Q}_p$-vector space realizing $\theta$, and $I_v$ is the inertial subgroup of a place of $\bar{\Q}$ above $v$.   Of course, by the inductive property of $L$-functions, we also have
\begin{equation}\label{lind}
L(\f_K,\theta,s)=L(f,\phi_{\theta},s),
\end{equation}
where $\phi_{\theta}$ is the Artin representation of $\Gal(\bar{\Q}/\Q)$ induced from the representation $\theta$ of $\Gal(\bar{\Q}/K).$

\medskip

On the other hand, the classical theory of Rankin products  (see \cite[\S 4]{S3}) considers instead the complex $L$-series ${\mathcal D}(\f_K,\g,s)$ defined by
\begin{equation}\label{20}
{\mathcal D}(\f_K,\g,s)=L_{\frak n}(\psi, 2s-k+1)\times \underset{\a}{\sum}\, c(\a)b (\a)N(\a)^{-s}
\end{equation}
with $\a$ running over all integral ideals of $K$; here $\n$ is the least common multiple of the levels of $\f_K$ and $\g$, $\psi$ is the character of $\g$, and $L_{\n}(\psi,s)$ is the imprimitive $L$-series of $\psi$ where the Euler factors at the primes dividing $\n$ have been omitted.
A well-known classical argument shows that ${\mathcal D}(\f_K,\g,s)$ has the Euler product expansion
\begin{equation}\label{21}
{\mathcal D}(\f_K,\g,s)=\underset{v}{\prod}\, D_v(\f_K,\g,(N v)^{-s})^{-1}
\end{equation}
where
\begin{equation}\label{22}
D_v(\f_K,\g,X)=\det \left(1-{\rm Frob}_v^{-1}X\mid \left( M_p(f)^{I_v}\otimes_{\bar{\Q}_p}W_{\theta}^{I_v}\right)\right).
\end{equation}
Thus the complex $L$-functions $L(\f_K,\theta,s)$ and ${\mathcal D}(\f_K,\g,s)$ coincide, except for the possible finite set of Euler factors  at places $v$ for which
\begin{equation}\label{21'}
\left( M_p(f)\otimes_{\bar{\Q}_p}W_{\theta}\right)^{I_v}\neq \left( M_p(f)^{I_v}\otimes_{\bar{\Q}_p}W_{\theta}^{I_v}\right).
\end{equation}

\medskip
\noindent To avoid this technical difficulty, we impose an additional simplifying hypothesis.

\medskip

\begin{lem}\label{pqf}
Assume that for each prime $q$ such that $q^2\mid N$, that $q$ does not divide the conductor of the representation of $\Gal(\bar{\Q}/\Q)$ induced from $\theta.$ Then for every prime number $p$, and
every finite place $v$ of ${\cK}$ which does not lie above $p$,  we have
\begin{equation}\label{mpf}
\left(M_p(f)\otimes_{\bar{\Q}_p} W_{\theta})\right)^{I_v}=(M_p(f))^{I_v}\otimes_{\bar{\Q}_p} (W_{\theta})^{I_v},
\end{equation}
where $I_v$ denotes the inertial subgroup at $v$. In particular, ${\mathcal D}(\f_K,\g,s)=L(\f_K,\theta,s).$
\end{lem}
\proof
Suppose that $v$ lies above a prime $q$, where $q\neq p$. Assume first  that $(q,N)=1.$ Then $I_q,$  and hence also $I_v$, acts trivially on $M_p(f),$ and so \eqref{mpf} is plain. Suppose next that $q$ divides $N$ but $q^2$ does not divide $N$. Then it is well known that the image of $I_q,$
hence also that of $I_v,$ in the automorphism group of $M_p(f)$ is infinite, and that $M_p(f)^{I_q}$ has dimension one over $\bar{\Q}_p.$  Clearly the same assertions remain valid if we replace $I_q$ by any open subgroup $I'_q$ of $I_q$. Thus we must have $M_p(f)^{I_v}=M_p(f)^{I'_v}$ for every open subgroup $I'_v$ of $I_v.$ Since some open subgroup of $I_v$ acts trivially on $W_{\theta}$, \eqref{mpf} follows immediately.
Finally, if $q^2$ divides $N$,  the hypothesis of the lemma  shows that $I_v$ acts trivially on $W_{\theta}$, whence \eqref{mpf} is again clearly true.

\qed

\medskip

Our next result relates the automorphic period of $\f_K$ to the periods $\Omega^+(f)$ and $\Omega^-(f).$ We normalize the Petersson inner product on the space of cusp forms of level dividing $N{\mathcal O}_K$ for the Hilbert modular group of $K$ as in \cite{S3} (see formula (2.7) on p. 651).

\begin{prop}\label{2.4}
Let $K$ be a real abelian field, and write $\f_K$ for the base change of $f$ to $K$. Then
\begin{equation}\label{17'}
\frac{(2\pi i)^{(1-k)\beta} \pi^{\beta k}\langle \f_K, \f_K \rangle_K}{\left(\Omega_+(f)\times \Omega_-(f)\right)^{\beta}} \in \Q,
\end{equation}
where $\beta=[K:\Q].$
\end{prop}

We shall use the following notation in the proof of this proposition. If $\psi$ is any abelian character of $K$, write $L(f/K,\psi,s)$ for the primitive $L$-function attached to the tensor product of  $\psi$ with the restriction of \eqref{3} to $\Gal(\bar{\Q}/K).$  Also, for any abelian character $\chi$ of $\Q$, we write $\chi_K$ for the restriction of $\chi$ to $\Gal(\bar{\Q}/K).$

\begin{lem}\label{2.5}
Let $K$ be any  real abelian extension of $\Q$ and $\eta$ any abelian character of $\Q$. Then there exists an abelian character $\chi$ of $\Q$ as follows. For all $\sigma$ in $\Gal(\bar \Q/\Q)$, we have
\begin{enumerate}
\item $ L(f/K,\chi_K^{\sigma},k/2)\neq 0~{\rm and}~ L(f/K, \chi_K^{\sigma}\eta_K, k/2)\neq 0;$
\item $ L(f/K,\chi_K^{\sigma},s) ~({resp.}~L(f/K,\chi_K^{\sigma}\eta_K,s))$ has Euler factor equal to 1 at all places of $K$ where $\chi_K^{\sigma}$ (resp. $\chi_K^{\sigma}\eta_K$) is ramified.
\end{enumerate}
\end{lem}
\proof
Let  $\Sigma$ be any finite set of primes of $\Q$ containing the primes dividing $N$, the primes dividing the conductor of $\eta$, and the primes which ramify in $K$. By an important theorem of Rohrlich \cite{R}, there exists a finite abelian extension $M$ of $\Q$, unramified outside $\Sigma$, such that $L(f,\lambda,k/2)\neq 0$ for every abelian character  $\lambda$ of $\Q$ that is unramified outside $\Sigma$, and which does not factor through $\Gal(M/\Q).$ By enlarging $M$ if necessary, we can assume that $M \supset K$. Let ${\mathcal N}_K$ be the conductor of $f/K$, and $\Delta_{M/K}$ the relative discriminant of $M$ over $K$. Also, if $\xi$ is an abelian character of $K$, write ${\mathcal N}(\xi)$ for its conductor. Let $\chi$ be any abelian character of $\Q$ such that, for every prime $v$ of $K$ above $\Sigma$, we have
\begin{equation}\label{23}
\ord_v({\mathcal N}(\chi_K)) > {\rm max}\big\{\ord_v(\Delta_{M/K}),\ord_v({\mathcal N}(\eta)),\ord_v({\mathcal N}_K)\big\}.
\end{equation}
Such a character can always be found by taking a character of $\Gal(\Q(\mu_m)/\Q)$ for sufficiently large $m$. Let $v$ be any place of $K$ above $\Sigma,$ and put
$$
t_v={\rm max}\big\{\ord_v(\Delta_{M/K}),\ord_v({\mathcal N}_K)\big\}.
$$
Thanks to \eqref{23}, it is clear, that for each $\sigma$ in  $\Gal(\bar{\Q}/\Q)$, we have
\begin{equation}\label{24}
\ord_v({\mathcal N}(\chi_K^{\sigma}))=\ord_v({\mathcal N}(\chi_K^{\sigma}\eta))>t_v.
\end{equation}
In particular, none of these characters can factor through $\Gal(M/K).$ Moreover, it is also easily seen from \eqref{24} that
$$
\left( M_p(f) \otimes_{\bar{\Q}_p}\chi_K^{\sigma}\right)^{I_v} =\left(M_p(f) \otimes_{{\bar{\Q}}_p}\chi_K^{\sigma}\eta\right)^{I_v}=0,
$$
whence the final assertion of the lemma is clear.
\qed

\medskip

We now prove that the left hand side of \eqref{17'} is an algebraic number. Take $J=K(i),$ and let $\eta$ be any abelian character of $\Q$ such that $J$ is the fixed field of the kernel of $\eta_K.$ Now let $\chi$ be an abelian character of $\Q$  having the properties specified in Lemma \ref{2.5}, and write $\chi_J$ for the restriction of $\chi$ to $\Gal(\bar{\Q}/J).$ Note that the representation of $\Gal(\bar{\Q}/K)$ induced by $\chi_J$ is $\theta=\chi_K \oplus \chi_K\eta_K.$ Write $\g$ for the Hilbert modular form relative to $K$ which corresponds to $\theta.$ Thus $\g$ has parallel weight one, and character $\eta_K\chi_K^2.$ Moreover, by the second assertion of Lemma 2.5, we have the exact equality of $L$-functions
\begin{equation}\label{25}
{\mathcal D}(\f_K,\g,s) =L(f/K,\theta,s).
\end{equation}
On the other hand, since $K$ is abelian over  $\Q$, we also have
the identity
\begin{equation}\label{26}
L(f/K,\theta,s)=\overset{\beta}{\underset{j=1}{\prod}}\,L(f,\chi\zeta_j,s)L(f,\chi\eta\zeta_j,s),
\end{equation}
where $\zeta_1,\dots,\zeta_d$ denote the characters of $\Gal(K/\Q).$

\medskip

The desired algebraicity assertion follows by evaluating both sides of \eqref{26} at $s=k/2$, noting that this common value is non-zero by Lemma \ref{2.5}, and then applying Shimura's algebraicity results to each $L$-function. Indeed, Theorem 4.2 of \cite{S3}
shows that
\begin{equation}\label{27}
\frac
{{\mathcal D}(\f_K,\g,k/2)}
{(2\pi i )^{\beta}\pi^{\beta k} \langle \f_K, \f_K \rangle \tau_K(\eta_K\chi_K^2) }
\in \bar{\Q},
\end{equation}
where $\tau_K(\eta_K\chi_K^2)$ denotes the Gauss sum for the character $\eta_K\chi_K^2$ of $\Gal(\bar{\Q}/K)$ (see (3.9) of \cite{S3} for the definition of this Gauss sum). On the other hand, recalling that $\chi\zeta_j(-1) \neq \chi\zeta_j\eta(-1)$ for $j=1,\dots,\beta,$ it follows from
\cite[Theorem 1]{S2} that
\begin{equation}\label{28}
\frac
{\overset{\beta}{\underset{j=1}{\prod}}\,L(f,\chi\zeta_j,k/2)\times L(f,\chi\eta\zeta_j,k/2)}
{(2\pi i)^{\beta k}\times(\Omega_+(f)\Omega_{-}(f))^{\beta}\times
\overset{\beta}{\underset{j=1}{\prod}}\,\tau_{\Q}(\chi\zeta_j)\tau_{\Q}(\chi\eta\zeta_j)
}
\in\bar{\Q};
\end{equation}
here $\tau_{\Q}(\kappa)$ denotes the usual Gauss sum of an abelian character $\kappa$ of $\Q$.
Combining \eqref{27}and \eqref{28}, it follows immediately that the left hand side of \eqref{17'} is an algebraic number. Moreover, a more detailed analysis, exactly as in the proof of
\cite[Theorem 3.4]{BD1} shows that this algebraic number is invariant under the action of $\Gal(\bar{\Q}/\Q)$, completing the proof of Proposition \ref{2.4}.
\qed

\medskip

As a first application of  Proposition \ref{2.4}, we establish the following case of Conjecture \ref{alg}.

\begin{thm}\label{2.7}
Assume $F$ is an imaginary number field with $\Gal(F/\Q)$ abelian. Let $\psi$ be any abelian character of $\Gal(\bar{\Q}/F)$, and let $\phi$ be the induced character of $\Gal(\bar{\Q}/\Q)$.
Assume that, for every prime $q$ such that $q^2$ divides $N$, $q$ does not divide the conductor of $\phi.$
Then Conjecture \ref{alg} is valid for $f$ and $\phi.$
\end{thm}
\proof
Let $K$ be the maximal real subfield of $F$, and let $\theta$  be the representation of $\Gal(\bar{\Q}/K)$ induced from $\psi.$ Thus $\theta$ is a two dimensional Artin representation of $\Gal({\bar {\Q}}/K)$, and we let ${\frak g}$ be the associated Hilbert modular form as above. Then, by Lemma \ref{mpf},
\begin{equation}\label{dl}
{\mathcal D}(\f_K,\g,s)=L(\f_K, \theta,s)=L(f,\phi,s).
\end{equation}
But, assuming $n$ is an integer with $1\leq n \leq k-1,$ it is shown in \cite[(4.10)]{S3} that
\begin{equation}\label{pb}
\frac{(2\pi i)^{-2n\beta}{\mathcal D}(\f_K,\g,n)}
{(2\pi i)^{\beta(1-k)}\pi^{\beta k}\langle \f_K,\f_k\rangle_K}
\in \bar{\Q}.
\end{equation}
Now making use of  Lemma \ref{2.5}, and noting that
$$
d(\phi)=2\beta,~~d_n^+(\phi)=d_n^-(\phi)=\beta,
$$
the algebraicity statement \eqref{15} follows on putting $s=n$ in \eqref{dl}. This completes the
proof of Theorem \ref{2.7}.
\qed

\medskip

Again following the ideas of  \cite{BD1}, we now prove a refined version of  Conjecture \ref{alg}
for Artin representations $\phi$ which factor through the Galois group over $\Q$ of the field
\begin{equation}\label{16}
F_r = \Q(\mu_{p^r},m^{1/p^r});
\end{equation}
here $p$ is an odd prime number, $r \geq 1$ is an integer, $\mu_{p^r}$
is the group of $p^{r}$-th roots of unity, and $m$ is an integer $>1.$
For simplicity, we shall always assume that $m$ is not divisible by the $p$-th power
of an integer $>1.$
In order to state the refinement
of  \eqref{15}, we first recall the epsilon-factors of the Artin representation $\phi$ (for a fuller discussion, see \cite[\S 6.2]{DD}).  Fix the Haar measure $\mu$ on $\Q_p$ determined by
$\mu(\Z_p) = 1$, and the additive character $\alpha$ of $\Q_p$ given by
$$
\alpha (zp^{-t}) = e^{2i{\pi}z/p^t}, \,  {\rm for} \,  z \in \Z_p.
$$
Write $\epsilon_p(\phi)$ for the local epsilon-factor of $\phi$ at the prime $p$, which is uniquely determined by this choice of $\mu$ and $\alpha$.
For each integer $n=1,\dots,k-1,$ define
\begin{equation}\label{17}
L_p^*(f,\phi,n)= \frac{ L(f,\phi,n) \epsilon_p(\phi)}{\left((2\pi i)^{nd(\phi)}\times \Omega_+(f)^{d^+_n(\phi)}\times|\Omega_{-}(f)|^{d^{-}_n(\phi)}\right)}
\end{equation}

\medskip

\noindent{\bf Hypothesis H2:}  For all primes $q$ such that $q^2$ divides $N$, we have $(q,mp)=1.$

\begin{thm}\label{2.3}
Assume that the Artin representation  $\phi$ factors through $\Gal(F_r/\Q)$ for some integer $r \geq 1$, where $F_r$ is given by \eqref{16}. Suppose in addition that Hypotheses H1 and H2 are valid. Then $\Lambda(f,\phi,s)$  is entire, and satisfies the functional equation \eqref{12}. Moreover, for all integers $n=1,\dots,k-1,$ $L_p^*(f,\phi,n)$ is an algebraic number satisfying
\begin{equation}\label{35}
L_p^*(f,\phi,n)^{\sigma}=L_p^*(f,\phi^{\sigma},n)
\end{equation}
for all $\sigma$ in $\Gal(\bar\Q/\Q).$
\end{thm}
\proof
As remarked above, the proof we now give follows closely that given in \cite{BD1}, where $f$ was assumed to have weight $k=2$, and therefore corresponded to an isogeny class of elliptic curves defined over $\Q.$ For each integer $r\geq 1$,  define ${\cK}_r=\Q(\mu_{p^r})$  and write $K_r$ for its maximal real subfield. Note that $\Gal(F_r/{\cK}_r)$ is cyclic of order $p^r,$  since $m$ is assumed to be $p$-power free. Put
\begin{equation}\label{36}
F_{\infty}=\bigcup_{r\geq 1} F_r,~~~~~~~~~{\cK}_{\infty}=\bigcup_{r \geq 1} {\cK}_r,~~~~~~~~~G=\Gal(\Fin/\Q).
\end{equation}
For this proof, define $\rho$ to be the representation of $\Gal(F_r/\Q)$ induced by any character of exact order $p^r$
of $\Gal(F_r/{\cK}_r).$ It is then easy to see that $\rho$ is irreducible, and that every irreducible Artin representation $\phi$ of $G$ is of the form $\lambda$ or $\rho\lambda$ for some integer $r\geq 1$,where $\lambda$ is a one dimensional character of $\Gal({\cK}_{\infty}/\Q)$. For the proof of  Theorem \ref{2.3}, we may assume that $\phi$ is irreducible. Now it is clear from these remarks that every irreducible Artin representation $\phi$  of $G$ is induced from an abelian character of  ${\cK}_r$ for some integer $r \geq 1$. Thus Theorem \ref{sha}  implies that
$\La(f,\phi,s)$ is entire and satisfies the functional equation \eqref{12}. Also, noting that $\Fin/\Q$ is unramified outside of the primes dividing $mp,$ we conclude from Theorem \ref{2.7} that $L_p^*(f,\phi,n)$ is an algebraic number for all Artin characters $\phi$ of $G$ and all integers $n=1,\dots, k-1.$ Thus it remains to establish \eqref{35} for irreducible $\phi$.

\medskip

If $d(\phi)=1$, one can easily deduce \eqref{35} from \cite[Theorem 1]{S1}. Assuming $d(\phi)>1,$ it follows that  for some integer $r\geq 1$, $\phi$ is induced by an abelian character of ${\cK}_r$ of the form $\psi \lambda_{{\cK}_r},$ where $\psi$ is a character of $\Gal(F_r/{\cK}_r)$ of exact order $p^r$, and
$\lambda_{{\cK}_r}$  is the restriction to $\Gal(\bar{\Q}/{\cK}_r)$ of a one dimensional character $\lambda$ of $\Gal({\cK}_{\infty}/\Q).$  We define $\theta$ to be the two dimensional Artin representation of $\Gal(\bar{\Q}/K_r)$ induced by $\psi\lambda_{{\cK}_r},$ and take $\g$ to be the corresponding Hilbert modular form relative to $K_r$ of parallel weight one.  Let $\nu$ be the abelian character of $K_r$ defining the quadratic extension ${\cK}_r/K_r,$  and let $\lambda_{K_r}$  be the restriction of $\lambda$ to $\Gal(\bar{\Q}/K_r).$  Since the determinant of $\theta$ is equal to $\nu \lambda_{K_r}^2$,  $\g$ will have character $\nu \lambda_{K_r}^2.$  Moreover, noting that Hypothesis H2 is valid for $\f$ and $\phi$ because the conductor of $\phi$ can only be divisible by primes dividing $mp,$ we conclude from Lemma \ref{pqf} that
\begin{equation}\label{37}
{\mathcal D}(\f_{K_r},\g,s)=L(f_{K_r},\theta,s)=L(f,\phi,s).
\end{equation}
Taking $s=n$ with $1 \leq n \leq k-1,$ it then follows from \cite[Theorem 4.2]{S3} that
\begin{equation}\label{38}
{\mathcal A}(f,\phi,n):=
\frac{L(f,\phi,n)}
{(2\pi i)^{d(\phi)(1+2n-k)/2}\times \pi^{kd(\phi)/2}\times \langle \f_K,\f_K \rangle_{K_r}
\times \tau_{K_r}(\nu \lambda_{K_r}^2)}
\end{equation}
satisfies
$$
{\mathcal A}(f,\phi,n)^{\sigma}={\mathcal A}(f,\phi^{\sigma},n)
$$
for all $\sigma$ in $\Gal(\bar{\Q}/\Q);$ here $\tau_{K_r}(\nu\lambda_{K_r}^2)$ is the Gauss sum of the abelian character $\nu\lambda_{K_r}^2$  of $K_r$, as defined by \cite[(3.9)]{S3}. Noting that
$$
d(\phi)=2[K_r:\Q],~d^+(\phi)=d^-(\phi)=[K_r:\Q],
$$
we conclude easily from \eqref{17'} and \cite[Proposition 4.5]{BD1} that the last assertion of Theorem \ref{2.3}
will hold if and only if
\begin{equation}\label{38'}
\left( \underset{q\neq p,\infty}{\prod}\, \epsilon_q(\phi)\right)^{\sigma}=\underset{q\neq p,\infty}{\prod}\,\epsilon_q(\phi^{\sigma}),
\end{equation}
for all $\sigma$ in $\Gal(\bar{\Q}/{\Q}).$ But \eqref{38'} is an immediate consequence of the fact that $\rho$ can be realized over $\Q$, and the equation
$$
\epsilon_q(\phi)=\epsilon_q(\lambda)^{\dim(\rho)}\,\sgn\left( \det(\rho)(q^{e_q(\rho)})\right),
$$
which holds for all $q \neq p,$ since $\lambda$ is unramified at $q$; here $e_q(\rho)=\ord_q({\mathcal N}(\rho))$, with ${\mathcal N}({\rho})$ denoting the conductor of $\rho.$ This completes the proof.
\qed

\section{Interlude on root numbers}

Recall that $F=\Q(\mu_p,m^{1/p}).$ From now on, write
$\rho$ for the unique irreducible representation of $\Gal(F/\Q)$ of dimension $p-1$;
it is induced from any non-trivial character of $\Gal(F/\mathcal K).$
We now describe the local root numbers
$$
  w_q(f,\rho) = \frac{\epsilon_q(f,\rho)}{|\epsilon_q(f,\rho)|}
$$
and the corresponding global root number $w(f,\rho)$
under the hypothesis H2 above.
This global root number is the sign in the functional equation
of the twisted $L$-function $L(f,\rho,s)$.
A similar computation in weight 2, i.e., for elliptic curves, was carried out by
V. Dokchitser \cite{VD}.

\begin{thm}
\label{rootrho}
Let $f=\sum_n a_n e^{2\pi inz}$ be a primitive cusp form of conductor $N$ with trivial character, and
weight $k \geq 2$. Assume that for all primes $q$ such that $q^2|N$, we have $(q,mp)=1$.
Then, for every finite prime $q$, the local root number $w_q(f,\rho)$ is  given by
$$
  w_q(f,\rho) = w_q(\rho)^2\times\left\{
    \begin{array}{llll}
       (\frac qp)^{\ord_q(N)} &\text{if}&(q, pm)=1, \cr
       -\sgn\>a_p          &\text{if}&q=p,\>\,\ord_p(N)=1\text{ and }
                            m^{p-1}\equiv 1\!\!\mod p^2,\cr
       1                   &\rlap{otherwise.}&\cr
    \end{array}
  \right.
$$
Further, the global root number is given by
$$
  w(f,\rho) \>\>=\>\> (-1)^{\frac{p-1}2}\>\delta\>
      \prod_{q\nmid pm}\> \Bigl(\frac qp\Bigr)^{\ord_q(N)},
$$
where $\delta=-\sgn\>a_p$ when both $\ord_p(N)=1$ and $m^{p-1}\equiv 1\!\!\!\mod p^2$, and $1$ otherwise.
\end{thm}

\begin{proof}  Let $l$ be any prime distinct from $q$,
and as before, let $V_l$ be the $l$-adic Galois representation attached
to $f$.  Put $n(V)=\ord_q(N),$  and let $n(\rho)$ be such that $q^{n(\rho)}$
is the $q$-part of the conductor of $\rho$.
We note that the determinant $\det\,\rho$ of $\rho$ equals $(\frac\cdot p)$, the
non-trivial quadratic character of $\Gal(\Q(\mu_p)/\Q)$.
\noindent Recall that the inverse  local  Euler factors of $L(f,s)$ are
$$
  P_q(f,T) = \left\{
  \begin{array}{llll}
  1-a_qT+q^{k-1}T^2&&\text{if}& (q,N)=1,\cr
  1-a_qT&&\text{if}& \ord_q(N)=1,\cr
  1&&\text{if}& \ord_q(N)\geq 2.\cr
  \end{array}
  \right.
$$

\smallskip

\noindent The local root numbers $w_q(f, \rho)$ can be computed as follows:-

\noindent{\bf Case 1} ($(q,N)=1$):  In this case $V_l$ is unramified, and we can use
the unramified twist formula \cite[3.4.6]{TatN},
$$
  w_q(f,\rho)=w_q(\rho)^{\dim V_l}\cdot\sgn((\det \,V_l)(q^{n(\rho)})).
$$
Here $\sgn \,z=\frac{z}{|z|}$ for $z\in\C$, and we evaluate the one-dimensional
character $\det V$ on a number $q^{n(\rho)}\in\Q_q^\times$ via the local
reciprocity map. The second term is trivial since
$\det V_l$ is a power of a cyclotomic character which takes positive values
on~$\Q_q^\times$. Thus $w_q(f, \rho)=w_q(\rho)^2$, as asserted.

\smallskip

\noindent{\bf Case 2}  ($\ord_q(N)=1$):  Here $a_q\ne 0, \dim V_l^{I_q}=1$, the action of inertia
$I_q$ is unipotent $\sm1*01$, and the action of Frobenius is $\sm{a_q}*0{a_q^{-1}p^{k/2}}$,
where $k$ is the weight of $f$; the top left corner can be seen
e.g. from the local factor.
Write $(V_l\tensor \rho)^{ss}$ for the semi-simplification of $V_l\tensor \rho$.
Writing $\tau={\rm Frob}_q$,  the semi-simplification formula for $\epsilon$-factors \cite[4.2.4]{TatN} gives
$$
  \begin{array}{lllll}
  w_q(f, \rho)&=&w_q((V_l\tensor \rho)^{ss})\>
  \frac{\sgn\,\det(-\tau | ((V_l\tensor \rho)^{ss})^{I_q})}
       {\sgn\,\det(-\tau | (V_l\tensor \rho)^{I_q})}\cr
  &=& w_q(\rho\oplus \rho)
  \frac{\sgn\,\det(-a_q\tau | \rho^{I_q})\sgn\,\det(-a_q^{-1}p^{k/2}\tau | \rho^{I_q})}
       {\sgn\,\det(-a_q\tau | \rho^{I_q})}\cr
  &=& \sgn\,\det(-a_q\tau | \rho^{I_q})^{-1} \cr
  &=& w_q(\rho)^2 \sgn(-a_q)^{d_q} \sgn\,\det(\tau | \rho^{I_q})^{-1};
  \end{array}
$$
here $d_q$ denotes the dimension of $\rho^{I_q}.$
It remains to determine $\rho^{I_q}$ and the action of Frobenius on it.
Let $J=\Q(m^{1/p})$. There is an equality of $L$-functions
$$
  \zeta_J(s) = \zeta(s) L(\rho,s).
$$
By considering the ramification of $q$ in $J/\Q$ and comparing the
local factors at $q$, we find that
$$
  P_q(\rho,T) = \left\{
  \begin{array}{llll}
  1+\cdots+(\frac qp)T^{p-1}&&\text{if}&q\nmid pm,\cr
  1&&\text{if}&q|m,\cr
  1-T&&\text{if}&p=q\text{ and $m$ is a $p$th power in $\Z_p^\times$},\cr
  1&&\text{if}&p=q\text{ and $m$ is not a $p$th power in $\Z_p^\times$}.\cr
  \end{array}
  \right.
$$
In particular, $d_q$ is even (and so $\sgn\>(-a_q)^{d_q}=1$) in all but the third case,
and $\det(\tau | \rho^{I_q})=1$ in all but the first case; in the first case,
$$
  \det(\tau| \rho^{I_q}) = \det(\tau | \rho) = \Bigl(\frac qp\Bigr) =
    \Bigl(\frac qp\Bigr)^{\ord_q(N)},
$$
as asserted by the formula. Finally, if $p\nmid m$, then it is easy to see by Hensel's lemma
that $m$ is a $p$th power in $\Z_p^\times$ if and only if it is a $p$th power in
$(\Z/p^2\Z)^\times$, which is in turn equivalent to the condition $m^{p-1}\equiv 1\mod p^2$.

\smallskip

\noindent{\bf Case 3} ($\ord_q(N)\geq 2$): By assumption, $q\nmid mp$, so $\rho$ is unramified. Then $w_q(\rho)=1$, and
by the unramified twist formula
$$
\begin{array}{lllll}
  w_q(f, \rho)&=&w_q(V_l)^{\dim \rho}\cdot\sgn((\det \rho)(q^{n(V)}))
   = (\pm 1)^{p-1}(\frac qp)^{n(V)} = (\frac qp)^{n(V)},\cr
\end{array}
$$
as claimed.

Turning to the global root number, we have
$$
  w(f,\rho) = \prod_v w_v(f, \rho),
$$
the product being taken over all places $v$ of $\Q$.  As $\rho$ is self-dual,
$$
  \prod_v w_v(\rho)^2 = w(\rho)^2 = 1,
$$
and the remaining contribution from the real place is
$(-1)^{\frac{p-1}2}$ (see e.g. \cite{VD}). This completes the proof.
\end{proof}

\begin{eg}
We compute the global root numbers $w(f,\rho)$ when $p=3$ and $f$ is one of the
primitive cusp forms with $(N,k)=(5,4), (5,6), (7,4)$ or $(121,4)$
that we will use in \S\ref{s:numdata} to illustrate the congruences.
In these cases, the answer does not actually depend on the weight.
\begin{itemize}
\item If $f$ has level $5$, then $\delta=1$ as $(3, N)=1$, whence
$$
  w(f,\rho) = (-1)^{\frac{3-1}2}\cdot 1\cdot \left\{
    \begin{array}{llll}
    (\frac53)&\text{if}&(5,m)=1\cr
    1        &\text{if}&\ord_5(m) \geq 1
    \end{array}
  \right.
  =
  \left\{
    \begin{array}{llll}
    1 &\text{if}&(5,m)=1\cr
    -1  &\text{if}&\ord_5(m) \geq 1.
    \end{array}
  \right.
$$
\item
Similarly, if $f$ has level 7, then $(\frac53)=-1$ is replaced by $(\frac73)=+1$
and we get $w(f,\rho)=-1$ for every $m$.
(cf. also \cite[\S7.1]{VD}, first example).
\item Finally, if $f$ has level 121, then $(\frac53)$ is replaced by
$(\frac{11}3)^{\ord_{11}(121)}=+1$, and we again get $w(f,\rho)=-1$ for every $m$.
\end{itemize}
The congruence that we verify involves the twists of $f$ by $\rho$ 
and by the regular representation $\sigma$ of $\Gal(K/\Q)\iso(\Z/p\Z)^\times$. 
It easy to check that the root numbers $w_q(f,\sigma)$ and $w(f,\sigma)$ are given by the formula in Theorem \ref{rootrho} with $m=1$. 
When $p\nmid N$ the formula becomes
$$
  w(f,\sigma) = (-1)^{\frac{p-1}2} \prod_{q|N} \Bigl(\frac qp\Bigr)^{\ord_q(N)} = 
     (-1)^{\frac{p-1}2} \Bigl(\frac Np\Bigr).
$$
In particular, for $p=3$ the global root number $w(f,\sigma)$ 
is $+1$ for the form of level 5, and $-1$ for the forms of level 7 and 121.
\end{eg}

\section{An analogue of a result of Hachimori-Matsuno}

The aim of this section is to establish an analogue for our primitive cusp form $f$ of results of Hachimori-Matsuno \cite{H-M}  for elliptic curves, over the fields

\begin{equation}\label{39}
{\mathcal K}_{\infty}=\Q(\mu_{p^{\infty}}),~~~F^{cyc}=\Q(\mu_{p^{\infty}},m^{1/p}),
\end{equation}
where again $m$ is an integer $>1$ which is $p$-power free. Such a result has already been established in \cite{PW}, but we wish to give a slightly more explicit result in order to explain its connexion with the congruence \eqref{1'}.
Write $\chi_p$ for the character giving the action of $\Gal({\bar \Q}/\Q)$ on $\mu_{p^{\infty}}.$ As usual,
for each $n \in \Z,$ write $\Z_p(n)$ for the free $\Z_p$-module of rank one on which $\Gal(\bar{\Q}/\Q)$ acts via
$\chi_p^n.$ If $W$ is any $\Gal(\bar \Q/\Q)$-module, which is also a $\Z_p$-module, define $W(n)=W \otimes_{\Z_p}\Z_p(n)$, endowed with the natural diagonal action of $\Gal(\bar{\Q}/\Q).$

\medskip

Let $V_p$ be the underlying $\Q_p$-vector space of the Galois representation $\tau_p$ attached to $f$.
Fix  once and for all a $\Z_p$-lattice $T_p$ in $V_p,$ which is stable under the action of $\Gal(\bar{\Q}/\Q).$
We stress that we always view $V_p$ as  the cohomology group, not the homology group of the motive $M(f).$
We assume from now on that $p$ and $f$ satisfy:-
\medskip

\noindent{\bf Hypothesis H3:} The odd prime $p$ is good ordinary for $f$, i.e., $p$ is an odd prime such that $(p,N)=(p,a_p)=1.$

\medskip

\noindent As $p$ is a good ordinary prime, it is shown in \cite{MW} that there exists a one dimensional subspace $V_p^0$ of $V_p$ such that the inertial subgroup of $\Gal(\bar{\Q}_p/{\Q_p})$ acts on $V_p/V_p^0$ by $\chi_p^{1-k}.$  Hence if we define
\begin{equation}\label{40}
A_{p^{\infty}}=V_p(k-1)/T_p(k-1),
\end{equation}
and define $A^0_{p^{\infty}}$ to be the image of $V_p^0(k-1)$ in $A_{p^{\infty}},$ then
$A_{p^{\infty}}/A^0_{p^{\infty}}$ is unramified at $p$.
For each finite extension ${\mathcal F}$ of $\Q$, define ${\mathcal F}^{\cyc}$ to be the  cyclotomic $\Z_p$-extension of ${\mathcal F}$, i.e., the compositum of ${\mathcal F}$ with the cyclotomic $\Z_p$-extension  of $\Q.$ We follow Greenberg and define the Selmer group of $A_{p^{\infty}}$ over ${\cF}^{\cyc}$ by
\begin{equation}\label{41}
\sel(A_{p^{\infty}}/{\cF}^{\cyc})=\ker\left(H^1({\cF}^{\cyc},\ap) \to\,\underset{w\nmid p}{\prod}\, H^1({\cF}^{\cyc}_w,\ap) \times \underset{w\mid p}{\prod}\, H^1({\cF}^{\cyc}_w,\ap/A^0_{p^{\infty}})\right),
\end{equation}
where $w$ runs over all finite places of $\fc$, and $\fc_w$ denotes the union of the completions at $w$ of the finite extensions of  $\Q$ contained in $\fc.$ Write
\begin{equation}\label{42}
X(\ap/\fc)=\Hom(\sel(\ap/\fc),\Q_p/\Z_p)
\end{equation}
for the compact Pontryagin dual of $\sel(\ap/\fc).$ Assuming  $\cF$ is Galois over $\Q$, both $\sel(\ap/\fc)$ and $X(\ap/\fc)$ are endowed with canonical left actions of $\Gal(\fc/\Q),$ and these extend by continuity to left module structures over the Iwasawa algebra
$$
\La(\Gal(\fc/\Q))=\underset{\leftarrow}{{\rm lim}}\,\Z_p[\Gal(M/\Q)],
$$
where $M$ runs over the finite Galois extensions of $\Q$ contained in $\fc.$

\medskip

\noindent   We shall need the following fundamental result of Kato (see \cite{K}). Note that for ${\cK}={\Q}(\mu_p)$, we have ${\kc}=\Q(\mu_{p^{\infty}}).$

\begin{thm}\label{3.1}
Assume Hypothesis H3. Then $X(\ap/\kc)$ is a torsion $\La(\Gal(\kc/\Q))$-module.
\end{thm}

\medskip

\noindent  Theorem \ref{3.1} implies that  the quotient
\begin{equation}\label{43}
X(\ap/\kc)/(X(\ap/\kc)(p))
\end{equation}
 is a finitely generated $\Z_p$-module, where $X(\ap/\kc)(p)$ denotes the  $p$-primary submodule. Define $\lambda(f/F^{\cyc})$ to be the $\Z_p$-rank of \eqref{43}.  We shall also need to consider the Euler factors of the complex $L$-function $L(f/{\cK},s)$ at places $v$ with $(v,Np)=1.$
 Let $q_v$ denote the characteristic of the residue field of $v$, and write $q_v^{r_v}$ for the absolute norm of $v$. Then these Euler factors are given explicitly by
 \begin{equation}\label{44'}
 P_v(f/{\cK},X)=\det(1-{\rm Frob}_v^{-1}X\mid V_p)= 1-b_vX+q_v^{r_v(k-1)}X^2,
 \end{equation}
 where ${\rm Frob}_v={\rm Frob}_{q_v}^{r_v},$ and $b_v\in \Z.$
 Since $q_v^{r_v}\equiv 1\,{\rm mod} \,p$, it is clear that for all integers $n$,
 we have
 \begin{equation}\label{44}
 P_v(f/{\cK}, q_v^{-r_vn})\equiv 2-b_v\,{\rm mod}\,p,
 \end{equation}
 when both sides are viewed as elements of $\Z_p.$
 In particular the question whether or not the left hand side lies in $p\Z_p$ is independent of $n$.
Define ${\mathcal P}_2$ to be the set of all places $w$ of ${\mathcal K}^{\cyc}$ such that,  writing
$v=w\mid {\cK}$, we have
\begin{equation}\label{45}
{\mathcal P}_2=\{ w:\,(q_v,Np)=1,~q_v\mid m,~{\rm and}
~ \ord_p(2-b_v)>0\}.
\end{equation}

\noindent  Similarly, suppose $v$ is a place of ${\cK}$, with residue characteristic
$q_v\neq p$ and $\ord_{q_v}N=1.$ Then the Euler factor $P_v(f/{\cK},X)$ is given explicitly by
\begin{equation}\label{45'}
P_v(f/{\cK}, X)=\det(1-{\rm Frob}_v^{-1}X\mid V_p^{I_v})=1-b_vX,
\end{equation}
where $b_v= a_{q_v}^{r_v},$ with $q_v^{r_v}$ again being the absolute norm of $v.$
Note again that
$$
P_v(f/{\cK},q_v^{-r_vn})\equiv 1-b_v\,{\rm mod}\, p
$$
for all integers $n$. Also, since
$a_{q_v}^2=q_v^{k-2}$ and $q_v^{r_v}\equiv 1 \,{\rm mod}\, p$, we always have $b_v^2\equiv 1\,{\rm mod}\, p.$
Define ${\mathcal P}_1$ to be the set of all places $w$ of $\kc$ such that, writing $v=w\mid{\cK}$, we have
\begin{equation}\label{46}
{\mathcal P}_1 =\left\{w:\, \ord_{q_v}N=1,~q_v\mid m~{\rm and }~ b_v\equiv1~{\rm mod}~p\right\}
\end{equation}

\medskip

To establish an analogue
for $f$ of the theorem of Hachimori-Matsuno, we shall need the following additional hypothesis.
 \medskip

\noindent{\bf Hypothesis H4:}  $X(\ap/\kc)$ is a finitely generated $\Z_p$-module.

\medskip

\noindent Recall that  ${\kc}=\Q(\mu_{p^{\infty}})$ and $\fc=\Q(\mu_{p^{\infty}}, m^{1/p}).$
\begin{thm}\label{3.2}
Assume Hypotheses H2, H3 and H4. Then $X(\ap/F^{\cyc})$ is also a finitely generated $\Z_p$-module and
\begin{equation}\label{45''}
\lambda(f/F^{\cyc})=p\lambda (f/\kc)+\sum_{w\in{\mathcal P}_2}\,
2(p-1) + \sum_{w\in{\mathcal P}_1}\,(p-1).
\end{equation}
\end{thm}

\proof

Put $\Delta=\Gal(F/{\cK})=\Gal(F^{\cyc}/\kc).$ If $B$ is any $\Delta$-module, we recall that the Herbrand quotient $h_{\Delta}(B)$ is defined by
$$
h_{\Delta}(B)=\frac{\#\,H^2(\Delta,B)}{\#\,H^1(\Delta,B)},
$$
whenever the cohomology groups are both finite.

Entirely similar arguments to those given for elliptic curves in \cite{H-M} show that, under the hypotheses H2, H3 and H4, $X(\ap/F^{\cyc})$ is indeed a finitely generated $\Z_p$-module, and we have
\begin{equation}\label{48}
\lambda(f/F^{\cyc})=p\lambda (f/\kc)+(p-1)\ord_p(h_{\Delta}(\sel(\ap/F^{\cyc}))
\end{equation}
where $h_{\Delta}(\ap/F^{\cyc})$ is finite.

\medskip

Let $\Sigma$ denote the set of primes of $\cK^{\cyc}$ lying above the rational primes
dividing $Nmp.$ As in \cite[\S 4]{H-M}, well-known arguments from Galois cohomology show that
\begin{equation}\label{49}
h_{\Delta}(\sel(\ap/F^{\cyc}))=\underset{w\in\Sigma}{\prod}\,h_{\Delta}\left(
\underset{u\mid w}{\prod}\,H^1(F_u^{\cyc},C_w)\right)
\end{equation}
where $u$ runs over the places of $F^{\cyc}$ above $w$, and
\begin{equation}\label{50}
C_w=\ap,~{\rm or}~\ap/\apo
\end{equation}
according as $w$ does not or does lie above $p$. Moreover, since  a prime $w$ of $\kc$ either splits completely or has a unique prime above it in $F^{\cyc},$ it is clear that the right hand side of \eqref{49} simplifies to a product of the $h_{\Delta}(H^1(F_u^{\cyc},C_w))$, where $w$ now runs over the primes in $\Sigma$ which do not split completely in $F^{\cyc}.$  Assume from now on that $w$ is a prime of $\kc$ which does not split in ${F}^{\cyc}$.  In particular, this means that the residue characteristic $q_w$ of $w$ must divide $pm$. Since $F_u^{\cyc}$ and ${\cK}_w^{\cyc}$  contain  $\mu_{p^{\infty}},$ their absolute Galois groups have $p$-cohomological dimension at most 1. As $\Delta$ is cyclic of order $p$, it then follows easily from the Hochschild-Serre spectral sequence that
\begin{equation}\label{50'}
H^i(\Delta, H^1(F_u^{\cyc},C_w))\simeq H^i(\Delta, C_w(F_u^{\cyc})),
\end{equation}
where $C_w(F_u^{\cyc})=H^0(F_u^{\cyc},C_w).$

\begin{lem}\label{3.3}
Assume there is a unique prime $u$ of $F^{\cyc}$ above $p$, and put $w=u\mid {\kc}$. Then we have
$$
h_{\Delta}(H^1(F_u^{\cyc},C_w))=1.
$$
\end{lem}
\proof
Since $C_w$ is unramified, and $F_u^{\cyc}$ is a totally ramified extension of $\Q_p$, we have
$$
C_w(F_u^{\cyc})=H^0(\Gal(\bar{{\mathbb F}}_p/{\mathbb F}_p), C_w).
$$
But ${\rm Frob}_p$ acts on $C_w$ by multiplication by the $p$-adic unit root of $1-a_pX+p^{k-1}X^2.$ However, this unit root cannot be equal to 1 as it has complex absolute value $p^{(k-1)/2}.$ Hence
$C_w(F_u^{\cyc})$ must be finite, and thus has Herbrand quotient equal to 1.
\qed

\medskip

\noindent Write $q_w$ for the residue characteristic of $w$.

\begin{lem}\label{3.4}
Assume that $w$ is a prime of $\kc$ such that $(q_w,Np)=1$, and $q_w$ divides $m$. Then there is a unique prime $u$ of $F^{\cyc}$ above $w$, and  $h_{\Delta}(H^1(F_u^{\cyc},C_w))=p^{-2}$ if
$w\in{\mathcal P}_2$, and $h_{\Delta}(H^1(F_u^{\cyc},C_w))=1$ otherwise.
\end{lem}
\proof
The first assertion of the lemma is clear since $w$ must ramify in $F^{\cyc}$ because $q_w$ divides $m.$ As $(q_w,Np)=1,$ we know that the inertial subgroup $I_w$ of the absolute Galois group of $\Q_{q_w}$ acts trivially on $V_p.$ We claim that $I_w$ also acts trivially on $C_w=\ap.$ Indeed, we have an exact sequence of Galois modules
$$
0 \to T_p \to V_p\to C_w\to 0,
$$
whence one obtains the long exact sequence
\begin{equation}\label{51}
0\to T_p \to V_p \to C_w^{I_w} \to H^1(I_w,T_p) \to H^1(I_w,V_p).
\end{equation}
As the inertial action is trivial on $T_p$ and $V_p$, and $q_w\neq p,$ we see that
$$
H^1(I_w,T_p)={\rm Hom}(J_w,T_p),~~~H^1(I_w,V_p)={\rm Hom}(J_w,V_p),
$$
where $J_w$ is the Galois group of the unique tamely ramified $\Z_p$-extension of ${\cK}_w^{\cyc}.$ Thus the last map in \eqref{51} is injective, and so $C_w^{I_w}=C_w$ as claimed.

\medskip

Let  $v$ be the restriction of $w$ to ${\cK}$. We next show that $C_v({\cK}_v)\neq 0$ if and only if $v\in{\mathcal P}_2.$ Since $C_v$ is unramified, we have the commutative diagram with exact rows
$$
\begin{CD}
0 @>>>T_p @>>>V_p @>>> C_w @>>>0\\
& & @V{\delta_v}VV @V{\delta_v}VV @V{\delta_v}VV&&\\
0 @>>>T_p @>>>V_p @>>> C_w @>>>0
\end{CD}
$$
where $\delta_v$ is the map given by applying ${\rm Frob}_v -1.$ The characteristic polynomial of ${\rm Frob}_v$ acting on $V_p $ is  $X^2P_v(f/{\cK},X^{-1}).$ The roots of this polynomial have complex absolute value $q_v^{r_v(k-1)/2},$ and thus are distinct from 1. Hence the middle vertical map in the above diagram is an isomorphism. It follows from the snake lemma that $C_w(K_v)$ has order
equal to the cokernel of the left hand vertical map, which is equal to the exact power of $p$ dividing
$P_v(f/{\cK},1).$ But
$$
P_v(f/{\cK},1)=1-b_v+q_v^{r_v(k-1)} \equiv (2-b_v)\,{\rm mod}\, p,
$$
showing that  $C_w(K_v)\neq 0$ if and only if $w \in {\mathcal P}_2.$

\medskip

Our next claim is that $C_w=C_w({\kc})$ if and only if $C_w({\cK}_v)\neq 0.$ As $\Gal({\cK}_w^{\cyc}/{\cK}_v)$ is pro-$p$, Nakayama's lemma shows that $C_w({\cK}_v)=0$ implies that $C_w({\cK}_w^{\cyc})=0.$ Conversely, assume that $C_w({\cK}_v)\neq 0.$ We then assert that the extension
${\cK}_v(C_w)$ is a pro-$p$ extension of ${\cK}_v.$ To prove this, let $(C_w)_p$ be the kernel of
multiplication by $p$ on $C_w.$ It is easily seen that  the extension ${\cK}_v(C_w)/{\cK}_v((C_w)_p)$ is pro-$p$. On the other hand, choosing an ${\mathbb F}_p$-basis of $(C_w)_p$ in which the first element belongs to $C_w({\cK}_v),$ and noting that the determinant of $(C_w)_p$ is trivial because it is equal to $\omega ^{r_v(k-1)},$ where $\omega$ is the cyclotomic character mod $p$, it follows that the extension ${\cK}_v((C_w)_p)/{\cK}_v$ is a $p$-extension. Thus ${\cK}_v(C_w)$ is a pro-$p$ extension of ${\cK}_v,$ and it is unramified  as inertia acts trivially on $C_w.$ Hence we must have
${\cK}_v(C_w)={\cK}_w^{\cyc}.$

\medskip

It is now clear from  \eqref{50} that  $h_{\Delta}(H^1(F_u^{\cyc},C_w))=0$ if $w \not \in {\mathcal P}_2,$  and $h_{\Delta}(H^1(F_u^{\cyc},C_w))=p^{-2}$ if $w \in {\mathcal P}_2.$ This completes the proof of the lemma.
\qed

\begin{lem}\label{3.5}
Assume that $w$ is a prime of ${\kc}$ such that $\ord_{q_w} N=1$ and $q_w$ divides $m$. Then there is a unique prime $u$ of $F^{\cyc}$ above $w$, and $h_{\Delta}(H^1(F_u^{\cyc},C_w))=p^{-1}$ if $w \in{\mathcal P}_1,$ and $h_{\Delta}(H^1(F_u^{\cyc},C_w))=1$ otherwise.
\end{lem}
\proof
The first assertion is clear, since $w$ must ramify in $F^{\cyc},$ because $q_v$ divides $m$.
Again, let $v$ be the restriction of $w$ to ${\mathcal K}.$  Since $\ord_{q_w}N=1,$ we have
$$
P_v(f/{\cK},X)=1-b_vX
$$
where we recall that $b_v^2 \equiv 1\,{\rm mod}\, p$.  Let $W_p$ be the subspace $V_p^{I_p}$ of $V_p, $ so that $\Gal(\bar{\Q}_{q_v}/F_v)$ acts on $W_p$ via the unramified character $\eta$ with
$\eta({\rm Frob}_v)=b_v.$ Choosing a basis of $V_p$ with the first basis element being a basis of $W_p,$ the representation of $\Gal(\bar{\Q}_{q_v}/F_v)$ on $V_p$ must be of the form
$\left(\begin{matrix} \eta &*\\0 & \lambda \end{matrix}\right)$, where $\lambda$ is a character of $\Gal(\bar{\Q}_{q_v}/F_v).$ As the determinant of $V_p$ is the cyclotomic character to the power $(k-1),$ we conclude that $\lambda$ is also unramified.  Moreover, the image of the restriction of this representation to $\Gal(\bar{\Q}_{q_v}/{\cK}_v^{\rm nr})$ is infinite, where ${\cK}_v^{\rm nr}$ is the maximal unramified extension of ${\cK}_v.$ Since $\eta$ takes values in $\Z_p^{\times},$ it is clear that the restriction of $\eta$ to $\Gal({\cK}_v^{\rm nr}/{\cK}_w^{\cyc})$ is the trivial character if and only if $w\in {\mathcal P}_1.$ Similarly, writing $v'$ for the restriction of $u$ to $F$, and recalling that $F_{v'}/{\cK}_v$ is totally ramified, it follows that  the restriction of $\eta$ to $\Gal(F_{v'}^{\rm nr}/F_u^{\cyc})$ is the trivial character if and only if $w \in {\mathcal P}_1.$ One concludes easily that, if $w \not \in {\mathcal P}_1,$ then $C_w(F_u^{\cyc})$ must be finite, and if $w \in{\mathcal P}_1,$ then the divisible subgroup of $C_w(F_u^{\cyc})$ has $\Z_p$-corank  1. In view of \eqref{50}, the assertion of the lemma is now clear.

\qed

\medskip

Combining \eqref{48}, \eqref{49}, and Lemmas 3.3, 3.4 and 3.5, the proof of  Theorem \ref{3.2} is now complete.
\qed

\section{The congruence from non-commutative Iwasawa theory}

As before,  let
\begin{equation}\label{52}
F=\Q(\mu_{p},m^{1/p}),~~~~{\cK}=\Q(\mu_{p})
\end{equation}
where $p$ is an odd prime, and $m>1$ is an integer which is not divisible by the $p$-th power of an integer $>1.$ Assume throughout this section that Hypotheses H1, H2, and H3 are valid. Let $\phi$ be an Artin representation of $\Gal(F_{\infty}/\Q)$. For each integer $n=1,\dots, k-1,$ we recall that $L_p^*(f,\phi,n)$ is defined by \eqref{17}. By Theorem \ref{2.3}, we know that $L_p^*(f,\phi,n)$ is an algebraic number.  Very roughly speaking, the non-commutative $p$-adic $L$-function seeks to interpolate the numbers $L_p^*(f,\phi,n),$ as $\phi,$ and $n$ both vary. While there has been important recent progress on the study of these non-commutative $p$-adic $L$-functions for the Tate motive over totally real number fields (see \cite{Ka},\cite{RW}), very little is still known about their existence for other motives, including the motive attached to  our modular form $f$. In the present paper, we shall only discuss what is perhaps the simplest congruence between abelian $p$-adic $L$-functions, which would follow from the existence of a non-commutative $p$-adic $L$-function for the motive of $f$ over the field $F_{\infty}$. A specialization of this congruence for elliptic curves has been studied in the earlier paper \cite{DD}.

\medskip

To state this congruence, we must first make a canonical modification of the values $L_p^*(f,\phi,n),$ given by \eqref{17} following \cite{CFKSV}, \cite{FK}. Recall that  since $(p,a_pN)=1,$ the Euler factor
$$
P_p(f,X)=1-a_pX+p^{k-1}X^2
$$
can be  written as
\begin{equation}\label{53}
P_p(f,X)=(1-\alpha X)(1-\beta X),
\end{equation}
where $\alpha$ is a unit in $\Z_p$, and $\ord_p(\beta)=k-1.$ We shall also need the Euler factors of the complex $L$-series $L(\phi,s)$ of the Artin representation $\phi$, which are defined by
\begin{equation}\label{54}
P_q(\phi,X)=\det\left(1-{\rm Frob}_q^{-1}X\mid M_l(\phi)^{I_q}\right)
\end{equation}
where $l$ is any prime distinct from $q$. As before, let $d(\phi)$ be the dimension of $\phi$.
Moreover, writing  ${\mathcal N}(\phi)$ for the conductor of $\phi$, define
\begin{equation}\label{55}
e_p(\phi)=\ord_p({\mathcal N}(\phi)).
\end{equation}
Recall that $P_q(f,\phi,X)$ defined by
\eqref{10'}  is the Euler factor at the prime $q$ of the complex $L$-function $L(f,\phi,s)$. Recall also that, for $n=1,\dots k-1$,
\begin{equation}\label{cr}
L_p^*(f,\phi,n)= \frac{ L(f,\phi,n) \epsilon_p(\phi)}{\left((2\pi i)^{nd(\phi)}\times \Omega_+(f)^{d^+_n(\phi)}\times|\Omega_{-}(f)|^{d^{-}_n(\phi)}\right)}.
\end{equation}
We then define
\begin{equation}\label{56}
M_p(f,\phi,n)= \Gamma(n)^{d(\phi)}\times L_p^*(f,\phi,n)\times  P_p(f,\phi,p^{-n})\times \frac{P_p(\hat{\phi},p^{n-1}/\alpha)}{P_p(\phi,\alpha/p^n)}\times ({p^{n-1}}/{\alpha})^{e_p(\phi)},
\end{equation}
and
\begin{equation}\label{56''}
{\cL}_p(f,\phi,n)=M_p(f,\phi,n)\,\underset{q\neq p,\,q\mid m}{\prod}\,P_q(f,\phi,q^{-n}),
\end{equation}
where $q$ runs over the prime factors of $m$ distinct from $p$. It is these modified $L$-values, defined using the naive periods $\Omega_+(f)$ and $\Omega_-(f)$,  which we shall actually compute in a number of numerical examples.

\medskip

Secondly, in order to obtain $p$-adic $L$-functions which will in the end satisfy the main conjectures of Iwasawa theory, we may also have to adjust the naive periods $\Omega^+(f)$ and
$\Omega^-(f)$ by certain non-zero rational numbers.  Writing $\Omega_+^{\can}(f)$ and $\Omega_-^{\can}(f)$ for these canonical periods, we will have
\begin{equation}\label{56'}
\Omega_+^{\can}(f)=c_+(f)\Omega_+(f),~~\Omega_-^{\can}(f)=c_ -(f) \Omega_-(f)
\end{equation}
for certain non-zero rational numbers $c_+(f)$ and $c_-(f).$ It is then natural to define
\begin{equation}\label{57'}
{\cL}_p^{\can}(f,\phi,n)=c_+(f)^{-d_n^+(\phi)} c_-(f)^{-d_n^-(\phi)}{\cL}_p(f,\phi,n).
\end{equation}
It is these modified values ${\cL}_p^{\can}(f,\phi,n)$ which should satisfy the non-abelian congruences for the $p$-adic $L$-functions arising in the main conjectures.
However, in our present state of knowledge, we do not know in general how to determine $c_+(f)$ and $c_-(f)$ precisely. Nevertheless, as we shall now explain, the work of Manin on the $p$-adic $L$-function of $f$ for the extension ${\cK}_{\infty}/\Q$ provides some partial information on this question.

\medskip

\begin{thm}\label{man}
Let $\sigma$ be the sum of the irreducible characters of $\Gal({\cK}_1/\Q),$ where ${\cK}_1=\Q(\mu_p).$ If $L(f,\sigma,k/2)=0,$ then
${\cL}_p^{\can}(f,\sigma,n)$ belongs to $p\Z_p$ for $n=1,\dots, k-1.$
\end{thm}
\proof
Let $\chi_p$ be the character giving the action of $\Gal(\bar{\Q}/\Q)$ on $\mu_{p^{\infty}}$. Fix a topological generator $\gamma$ of $\Gal({\cK}^{cyc}/K_1)$ and put $u=\chi_p(\gamma).$ The work of Manin then shows \cite{M} that there exists a power series $g(T)$ in $\Z_p[[T]]$ such that
\begin{equation}\label{57''}
g(u^r-1)=M_p^{\can}(f,\sigma,k/2+r),
\end{equation}
for all integers $r$ with $-k/2+1 \leq r \leq k/2-1,$
and where
$$
M_p^{\can}(f,\sigma,n)=(c_+(f)c_-(f))^{(1-p)/2}\, M_p(f,\sigma,n).
$$
Here it is understood that the canonical periods are those for which we expect $g(T)$ to be a characteristic power series for the dual Selmer group of $f$ over ${\mathcal K}_{\infty}.$ Assuming that $L(f,\sigma,k/2)=0,$ it follows that
$$
g(0)=0,
$$
and so $g(u^n-1)\in p\Z_p$ for all integers $n$. The assertion of the theorem then follows on noting that
$$
\underset{q\neq p,\,q\mid m}{\prod}\,P_q(f,\sigma,q^{-n})
$$
lies in $\Z_p$ for all $n \in \Z.$ This completes the proof.
\qed

\medskip

\begin{eg}
{\rm Take $f$ to be the unique primitive eigenform of level 7 and weight 4, and  $p=3.$ Then $L(f,\sigma,2)=0.$  Moreover, we see from Table II in \S 6 that ${\cL}_3(f,\sigma,1)\in 3\Z_3.$  In view of Theorem \ref{man}, this strongly suggests that in this case, we must have $\ord_3(c_+(f))=\ord_3(c_-(f))=0.$}
\end{eg}

\begin{eg}
{\rm Take $f$ to be the complex multiplication form of  level 121 and weight 4, which is attached to the cube of the Gr\"ossencharacter of the elliptic curve over $E$ over $\Q$ given by the equation \eqref{cme}, of conductor 121 and with complex multiplication by the full ring of integers of the field $L = \Q(\sqrt{-11})$, and again take $p=3.$ Then $L(f,\sigma,2)=0$. However, we see from Table III in \S 6 that ${\cL}_3(f,\sigma,1)$ is a 3-adic unit when $m=3 ,\,7$ or 11. Hence the naive periods $\Omega^+(f)$ and $\Omega^-(f)$ cannot be the good periods, and at least one of $\ord_3(c_+(f))$ or $\ord_3(c_-(f))$ must be strictly less than zero. In fact, in this case we do know the canonical periods for $f$, since, for all good ordinary primes $p$ for $f$, we know the periods for which the relevant cyclotomic main conjecture for $f$ over $\cK^{cyc}$ is valid. This is because this cyclotomic main conjecture can easily be deduced from the main conjecture for $E$ over the field obtained by adjoining to $L$ the coordinates of all $p$-power division points on $E$;  and this latter main conjecture is proven for all good ordinary primes $p$ for $E$ by the work of Yager and Rubin. Invoking the Chowla-Selberg formula, we see easily that the explicit values of these canonical periods can be taken as follows. Let
$$
\Theta = \Gamma(1/11)\Gamma(3/11)\Gamma(4/11)\Gamma(5/11)\Gamma(9/11).
$$
Then
\begin{equation}\label{cp}
\Omega_+^{\can}(f) = \sqrt{11}\times \Theta^3/(2\pi)^9, \, \Omega_-^{\can}(f) = i\Theta^3/(2\pi)^9.
\end{equation}
Direct computations show that
\begin{equation}\label{cp1}
\Omega_+(f)/\Omega_+^{\can}(f) = 1/22, \, \Omega_-(f)/\Omega_-^{\can}(f) = 3,
\end{equation}
whence
\begin{equation}\label{cp2}
\ord_3(c_+(f)) = 0, \,  \ord_3(c_-(f)) =-1,
\end{equation}
precisely as required.}

\end{eg}

As in the Introduction,  let $\sigma$ be the Artin representation of dimension $(p-1)$ given by the direct sum of the one dimensional characters of $\Gal({\cK}/\Q).$  Define $\rho$ to be the representation of $\Gal(F/\Q)$ induced from any non-trivial degree one character of $\Gal(F/{\cK})$. Thus $\rho$ also has dimension $(p-1),$ and is easily seen to be irreducible (cf.  \cite{DD}). Moreover, both $\sigma$ and $\rho$ are self-dual, can be realized over $\Z$, and their reductions modulo $p$ are isomorphic.  Let $R = \Z_p[[T]]$ be the ring of formal power series in an indeterminate $T$ with coefficients in $\Z_p$.  As explained in the Introduction, the work of Manin \cite{M} establishes the existence of a power series $H(\sigma, T)$ in $R$ satisfying the interpolation property \eqref{1.} It is conjectured that there exists a power series $H(\rho, T)$ in $R$ satisfying the interpolation condition \eqref{1.1}.

\begin{conj}(Congruence Conjecture).\label{cc}
Assume Hypotheses H1, H2, H3. Then there exists a power series $H(\rho, T)$ in $R$ satisfying the interpolation property \eqref{1.1},
and we have the congruence of power series
\begin{equation}\label{cc1}
H(\rho, T) \, \, \equiv \, \, H(\sigma, T) \, \,{\rm mod} \, \, pR.
\end{equation}
\end{conj}

We are grateful to M. Kakde for pointing out to us that the congruence \eqref{cc1} is simply a special case of the congruences
predicted by Kato in \cite{K}, and we now briefly explain why this is the case. Assume for simplicity that Hypothesis H4 is also valid.
Recall that $G$ denotes the Galois group of $F_{\infty}$ over $\Q$, and write $\Lambda(G)$ for the Iwasawa algebra of $G$,
$S$ for the canonical Ore set in $\Lambda(G)$, which is defined in \cite{CFKSV}, and $\Lambda(G)_S$ for its localization at $S$.
In addition, define ${\frak G}_0 = \Gal({\mathcal K}_{\infty}/\Q)$, and for each integer $n \geq 1$, let ${\frak G}_n$  be the unique
closed subgroup of index $p^{n-1}$ in $\Gal(\Q^{\cyc}/\Q)$. Write $S_n$ for the canonical Ore set of \cite{CFKSV} in the Iwasawa algebra
$\Lambda({\frak G}_n)$.  In \cite{K}, Kato defines a canonical map
$$
{\theta}_{G,S} : K_1(\Lambda(G)_S) \to \underset{n \geq 0}{\prod}K_1(\Lambda({\frak G}_n)_{S_n}),
$$
and characterizes its image by a remarkable set of congruences which we do not state in detail here. In particular, writing ${\theta}_{G,S}(\alpha) = (\alpha_n)$ for any element $\alpha$ of  $K_1(\Lambda(G)_S)$, we always have
\begin{equation}\label{cc2}
N(\alpha_0) \equiv \alpha_1 \,  {\rm mod} \, p,
\end{equation}
where $N$ denotes the norm  map from $K_1(\Lambda({\frak G}_0)_{S_0})$ to $K_1(\Lambda({\frak G}_1)_{S_1})$. Now take
$\alpha$ to be the conjectural $p$-adic $L$-function for $f$ over $F_{\infty}$, which we denote by ${\zeta}(f/F_{\infty})$. Let us also identify
$\Lambda({\frak G}_1)$ with the formal power series ring $R = \Z_p[[T]]$ by mapping the fixed topological generator $\gamma$ of ${\frak G}_1$ to $1+T$. Then it follows essentially from the construction of  the map ${\theta}_{G,S}$ and the interpolation properties of these $p$-adic $L$-functions that we will have
$$
N({\zeta}(f/{\Fin})_{0})= H(\sigma, T),  \, {\zeta}(f/{\Fin})_{1}= H(\rho, T).
$$
Thus the congruence \eqref{cc1} is indeed just a special case of the congruence \eqref{cc2} of Kato, as claimed.

\medskip
As was pointed out in the Introduction, if we evaluate both sides of the congruence \eqref{cc1} at the appropriate points in $p\Z_p$, we deduce the following congruence of normalized $L$-values from \eqref{1.} and \eqref{1.1}:-

\begin{conj}\label{4.1}
Assume Hypotheses H1, H2 and H3. Then for all integers $n=1,\dots, k-1$, we have
\begin{equation}\label{57}
{\mathcal L}_p^{\can}(f,\rho,n) \equiv {\mathcal L}_p^{\can}(f,\sigma,n)\,{\rm mod}\, p.
\end{equation}
\end{conj}
\medskip

We end this section by  explaining how this latter congruence is intimately connected with Theorem \ref{3.2}. Let ${\mathcal P}_1$ and ${\mathcal P}_2$ be the set of places of ${\cK}_{\infty}=\Q(\mu_{p^{\infty}})$ defined by \eqref{46} and \eqref{45} respectively.
\begin{lem}\label{4.2}
Let $q$ be any rational prime with $q$ dividing $m$ and $(q,Np)=1.$  Then all primes of  ${\kc}$ above $q$  belong to ${\mathcal P}_2$ if and only if $\ord_p(P_q(f,\sigma,q^{-n}))>0$ for some integer $n$.\end{lem}
\proof
Let $q$ have exact order $r_q$ modulo $p$, and let $v$ be a prime of ${\cK}$ above $q.$ Then
one sees immediately that
$$
P_q(f,\sigma,X)=(1-b_vX^{r_q}+q^{(k-1)r_q}X^{2r_q})^{\frac{p-1}{r_q}},
$$
where $b_v$ is defined by \eqref{44'}. Since $q^{r_q}\equiv 1\,{\rm mod}\, p$, the assertion of the lemma is now plain from the definition of ${\mathcal P}_2.$
\qed

\begin{lem}\label{4.3}
Let $q$ be any rational prime not equal to $p$ such that $q$ divides $ m$ and $\ord_q(N)
=1$ .Then all primes of $\kc$ above $q$ belong to ${\mathcal P}_1$ if and only if $\ord_p(P_q(f,\sigma,q^{-n}))>0$ for some integer $n$.
\end{lem}
\proof
Let $q$ have exact order $r_q$ and let $v$ be a prime of ${\cK}$ above $q$.
Since $\sigma$ is unramified at $q$, one sees easily that
$$
P_q(f,\sigma,X)=(1-b_vX^{r_q})^{\frac{p-1}{r_q}},
$$
where $b_v$ is defined by \eqref{45'}, and hence the assertion of the lemma is clear.
\qed

\medskip

By the work of Manin, we always have $M_p^{\can}(f,\sigma,n)$ is in $\Z_p$
for $n=1,\dots, k-1.$ Hence we conclude
from Lemmas \ref{4.2} and \ref{4.3} that ${\cL}_p^{\can}(f,\sigma,n)\in p\Z_p$ if either ${\mathcal P}_1$ or ${\mathcal P}_2$ is non-empty. On the other hand, assuming Hypotheses H1-H4, Theorem \ref{3.2} shows that $X(A_{p^\infty}/F^{\cyc})$ is infinite if either ${\mathcal P}_1$ or ${\mathcal P}_2$ is non-empty. But  $X(A_{p^{\infty}}/F^{\cyc})$ is infinite if and only if its characteristic element as a $\La(\Gal(F^{\cyc}/F))$-module is not a unit in the Iwasawa algebra. But the main conjecture for $X(A_{p^{\infty}})$ predicts that the ${\cL}_p^{\can}(f,\rho,n)$ are all values of the characteristic power series of $X(A_{p^{\infty}}/F^{\cyc})$. Thus it would follow that ${\cL}^{\can}_p(f,\rho,n)\in p\Z_p$ if either ${\mathcal P}_1$ or ${\mathcal P}_2$ is non-empty, in accord with the Congruence Conjecture \ref{4.1}.

\section{Numerical data}
\label{s:numdata}

We refer the reader to Section 6 of \cite{DD} for a detailed discussion of how the computations are carried out in the case of a primitive form  of weight 2.  Entirely similar arguments (see \cite{TD}) apply to the  calculation of the  numerical values ${\cL}_p(f,\phi,n),$ for $n=1,\ldots, k-1,$ for our given primitive modular form $f=\overset{\infty}{\underset{n=1}{\sum}}\, a_nq^n$ of conductor $N$.  We do not enter into the details here, apart from listing the explicit Euler factors which occur for the primes dividing $pm$. As before, let
$$
{\cK}=\Q(\mu_p),~~~ F=\Q(\mu_p,m^{1/p}),
$$
where  $m$ is a $p$-power free integer $>1.$  As earlier, we write $\phi$ for either  the direct sum $\sigma$ of the one dimensional characters of $\Gal({\cK}/\Q)$ or the unique irreducible representation $\rho$ of dimension $p-1$ of  $\Gal(F/\Q)$, and note that both of these Artin representations are self-dual.  We suppose that $p$ is an odd prime number satisfying
$(p,a_p)=(p,N)=1.$  In addition, we assume that Hypothesis H2 holds. As earlier, let $P_p(\phi,X)$ denote the polynomial in $X$ giving the inverse Euler factor at $p$ of the complex $L$-series $L(\phi,s)$ of the Artin representation $\phi$, and $P_q(f,\phi,X)$  the polynomial giving the inverse Euler factor at a prime $q$ of the complex $L$-series $L(f,\phi,s).$

\medskip

\begin{lem}
We have that $P_p(\sigma, X) = 1-X$, and $P_p(f,\sigma,X)=P_p(f,X)$.  If  $m \equiv \pm 1\, {\rm mod} \, \,p^2$,  then $P_p(\rho, X) = 1- X$, and $P_p(f,\rho,X)=P_p(f,X)$. Otherwise, both $P_p(\rho,X)$ and $P_p(f,\rho,X)$ are equal to 1.
\end{lem}
\qed

\begin{lem}
Let $q$ be any prime factor of $m$ distinct from $p$, and  write $r_q$ for the order of $q$ modulo $p$. Then we have:-
\begin{enumerate}
\item
$P_q(f,\sigma,X)=P_v\left(f/{\cK},X^{r_q}\right)^{\frac{p-1}{r_q}},$ where $P_v(f/{\cK},X)$ is the Euler factor of $f$ over ${\cK}$  at any prime $v$ of ${\cK}$ above $q$ if  $(q,N)=1.$
\item
$P_q(f,\sigma,X)=\left(1-a_q^{r_q}X^{r_q}\right)^{\frac{p-1}{r_q}}$  if  $\ord_q(N)=1.$
\item
$P_q(f,\rho,X)=1$.
\end{enumerate}
\end{lem}
\qed

\medskip

\noindent We remark that the computations require knowledge of the Fourier coefficients $a_n$ of $f$ for $n$ ranging from 1 up to approximately the square root of the conductor of the complex $L$-function $L(f, \phi, s)$. Since these conductors are very large even for small $N$, this explains why we need to know explicitly the $a_n$ for $1 \leq n \leq 10^8$, and why we are essentially restricted to the case  of the prime $p=3$. For our primitive cusp form $f$ of small conductor, we computed these Fourier coefficients
$a_n$ using \cite{sage} as follows.  We use linear algebra to express $f$
explicitly as a polynomial in terms of Eisenstein series (we only used
small conductor forms $f$ where this was possible), then we evaluate this expression using
arithmetic with polynomials of large degree over the integers.  This
high precision evaluation took about 1 day of CPU time in some cases,
and relies on the fast FFT-based polynomial
arithmetic from \url{http://flintlib.org}, and optimized code for
computing coefficients of Eisenstein series due to Craig Citro,
along with other optimizations specific to this problem.  For evaluation
of the CM form of level $121$ and weight $4$, we computed the Fourier
coefficients $d_p$ for the corresponding elliptic curve of weight $2$
(using \cite{P}), then obtained the coefficients $a_p$ of the weight $4$ form
as the sum of the cubes of the roots of $X^2 - d_pX + p$, and finally extended these multiplicatively to obtain all of the coefficients~$a_n$.

\medskip

For even $k \geq 2$, let $E_k(q)\in \Q[[q]]$ denote the weight $k$ Eisenstein series of level 1, normalized
so that the coefficient of $q$ is $1$.
For integers $t\geq 1$, define $E_2^*(q^t) = E_2(q) - tE_2(q^t)$, which is
a holomorphic modular form of level $t$ and weight $2$.
We consider 4 explicit primitive forms; 3 have expressions
in terms of Eisenstein series, and the fourth in terms of an
elliptic  curve with complex multiplication.  The first 3 are the unique primitive forms
on $\Gamma_0(p)$ with given weight. The fourth form  $f$ is the complex multiplication form of conductor 121 which is attached to the cube of the Grossencharacter of the elliptic curve
\begin{equation}\label{cme}
y^2 + y = x^3 - x^2 - 7x + 10.
\end{equation}
This curve has complex multiplication by the full ring of integers of ${\Q}(\sqrt{-11})$, and has conductor 121 when viewed as a curve over $\Q$. The following table gives the first few terms of the $q$-expansion of these four forms, and note that, in each case, 3 is an ordinary prime because the coefficient of $q^3$ is not divisible by 3.

\begin{center}
\small
\renewcommand{\arraystretch}{2}
\begin{tabular}{|c|c|l|}\hline
{\bf Conductor} & {\bf Weight} &  {\bf Primitive form} \\\hline
5  & 4 & $-\frac{250}{3} E_{4}(q^{5}) - \frac{10}{3} E_{4}(q) + 13
E_{2}^*(q^{5})^{2} = q - 4q^{2} + 2q^{3} + 8q^{4} - 5q^{5} + \cdots$\\\hline

7 & 4 & $-\frac{147}{2} E_{4}(q^{7}) - \frac{3}{2} E_{4}(q) + 5
E_{2}^*(q^{7})^{2} =
q - q^{2} - 2q^{3} - 7q^{4} + 16q^{5} + \cdots
$ \\\hline
5 & 6 &
$
\frac{521}{6} E_{6}(q^{5}) - \frac{1}{30} E_{6}(q) + 248
E_{2}^*(q^{5})E_{4}(q^{5})
= q + 2q^{2} - 4q^{3} - 28q^{4} + 25q^{5} + \cdots$
\\\hline
121 & 4 &
$q + 8q^{3} - 8q^{4} + 18q^{5} + 37q^{9} - 64q^{12} + 144q^{15} +
64q^{16} + \cdots $
\\\hline
\end{tabular}
\end{center}
\renewcommand{\arraystretch}{1}

\bigskip

The first two tables below provide numerical evidence in support of the congruences \eqref{1'}, and the third and fourth table below provides evidence in support of the stronger congruence \eqref{1''}.  The notation used in these four tables is as follows.  We have taken $p=3$, and assume that $\phi$ denotes either $\sigma$ or $\rho$, so that $d(\phi)=2$.
For each integer $n=1,\dots, k-1,$ put
\begin{equation}
L_3^*(f,\phi,n)= L(f,\phi,n) \epsilon_3(\phi)(2\pi i)^{-2n} (\Omega^+(f)|\Omega^-(f)|)^{-1},
\, \, \, \mathcal{P}_3(f, \phi, n) = \underset{q\mid 3m}{\prod}\,P_q(f,\phi,q^{-n}),
\end{equation}
and define
$$
\mathcal{L}_3(f, \phi, n) = \Gamma(n)^2\times L_3^*(f,\phi,n) \times \mathcal{P}_3(f, \phi, n) \times \frac{P_3(\phi,3^{n-1}/\alpha)}{P_3(\phi,\alpha/3^n)}\times ({3^{n-1}}/{\alpha})^{e_{3}(\phi)}.
$$
We also write $N(f, \phi)$ for the conductor of the complex $L$-function $L(f, \phi, s)$.
it is easily seen that $\epsilon _{3}(\sigma)$ is equal to the positive square root of $3$.  Moreover,
$\epsilon _{3}(\rho) = 3^5$ when $\ord_3(m) \geq 1$. When $(3, m) =1$, we have that $\epsilon _{3}(\rho)$ is equal to $3$ when $m \equiv \pm 1\,  {\rm mod} \, 3^2$,  and is equal to $3^3$ otherwise. If $r$ is any integer $\geq 1$,  and $w$ is an integer, the symbol $w + O(3^r)$ will denote a 3-adic integer  which is congruent to $w$ modulo  $3^r$.

\medskip

The reader should also bear in mind the following comments about the signs of the values $L_3^*(f, \phi, n)$ given in our tables below. Since $\phi$ can be realized over $\Q$, it follows from the convergence of the Euler product that $L(f, \phi, n)$ is strictly positive for $n=k/2+1,\dots, k-1$; in addition, the generalized Riemann hypothesis would also imply that the value at $n=k/2$ should either be zero or strictly positive (and this is the case in all of our numerical examples) Thus, by Theorem \eqref{2.3}, $L_3^*(f,\phi,n)$ is a rational number,which will have the sign $(-1)^nw(f,\phi)$ for $n=1,\dots, k/2-1$ by the functional equation \eqref{12}; and the sign of $L_3^*(f,\phi,k/2)$ should be $(-1)^{k/2}$ if it is non-zero.

\medskip

Finally, we recall (see Example 5.3 in section 5) that, for the form $f$ of conductor 121 and weight 4, the periods in Table IV are the naive periods, and that they must be replaced by the canonical periods defined in Example 5.3 to deduce the stronger congruence \eqref{1''} in this case.

\medskip

\advance\textheight by 2mm

\begin{center}
\tablefirsthead{%

 \hline
\multicolumn{7}{|c|}{Table I:   form $f$ of conductor 5 and weight 4.}\\
 \multicolumn{7}{|c|}{$L_3^*(f,\sigma,1)=-100,~~L_3^*(f,\sigma,2)=\frac{13}{3}$.}\\

 \hline  ${\tiny
\begin{array}{c}m\end{array}
}$ & ${\tiny
\begin{array}{c}L_3^*(f,\rho,n)\end{array}
}  $&${\tiny
\begin{array}{c}\mathcal{P}_3(f,\rho,n)\end{array} }$& ${\tiny
\begin{array}{c}\mathcal{P}_3(f,\sigma,n)\end{array}
}$ &${\tiny
\begin{array}{c}N(f,\rho)\end{array}
}
  $&${\tiny
\begin{array}{c} \mathcal{L}_3(f,\sigma,n)\end{array}
}$&${\tiny
\begin{array}{c} \mathcal{L}_3(f,\rho,n)\end{array}
}$ \\\hline }

 \tablehead{\hline \multicolumn{7}{|c|}{Table I:   form $f$ of conductor 5 and weight 4.}\\

\hline  ${\tiny
\begin{array}{c}m\end{array}
}$ & ${\tiny
\begin{array}{c}L_3^*(f,\rho,n)\end{array}
}  $&${\tiny
\begin{array}{c}\mathcal{P}_3(f,\rho,n)\end{array} }$& ${\tiny
\begin{array}{c}\mathcal{P}_3(f,\sigma,n)\end{array}
}$ &${\tiny
\begin{array}{c}N(f,\rho)\end{array}
}
  $&${\tiny
\begin{array}{c} \mathcal{L}_3(f,\sigma,n)\end{array}
}$&${\tiny
\begin{array}{c} \mathcal{L}_3(f,\rho,n)\end{array}
}$ \\\hline }

\tabletail{\hline}

 \tablelasttail{\hline}
\begin{supertabular}{|l@{\hspace{0pt}}|l@{\hspace{0pt}}|l@{\hspace{0pt}}|l@{\hspace{0pt}}|l@{\hspace{0pt}}|l@{\hspace{0pt}}|l@{\hspace{0pt}}|}
&$n=1$&&&&&\\
$  {\tiny
\begin{array}{c}2\end{array}
}  $ & $  {\tiny
\begin{array}{c}-2^{5}\cdot 5^{3}\cdot 7\cdot 13\end{array}
}  $ & $  {\tiny
\begin{array}{c}\frac{2\cdot 5^{2}}{3}\end{array}
}  $ & $  {\tiny
\begin{array}{c}1\end{array}
}  $ & $  {\tiny
\begin{array}{c}2^{4}\cdot 3^{6}\cdot 5^{2}\end{array}
}  $ & $  {\tiny
\begin{array}{c}1+O(3)\end{array}
}  $ & $  {\tiny
\begin{array}{c}1+O(3)\end{array}
}  $ \\
$  {\tiny
\begin{array}{c}3\end{array}
}  $ & $  {\tiny
\begin{array}{c}-2^{4}\cdot 5^{4}\cdot 13\cdot 41\end{array}
}  $ & $  {\tiny
\begin{array}{c}\frac{2\cdot 5}{3}\end{array}
}  $ & $  {\tiny
\begin{array}{c}1\end{array}
}  $ & $  {\tiny
\begin{array}{c}3^{10}\cdot 5^{2}\end{array}
}  $ & $  {\tiny
\begin{array}{c}2+O(3)\end{array}
}  $ & $  {\tiny
\begin{array}{c}2+O(3)\end{array}
}  $ \\
$  {\tiny
\begin{array}{c}5\end{array}
}  $ & $  {\tiny
\begin{array}{c}2^{5}\cdot 3\cdot 5^{2}\cdot 13\cdot 17\end{array}
}  $ & $  {\tiny
\begin{array}{c}0\end{array}
}  $ & $  {\tiny
\begin{array}{c}1\end{array}
}  $ & $  {\tiny
\begin{array}{c}3^{6}\cdot 5^{4}\end{array}
}  $ & $  {\tiny
\begin{array}{c}0\end{array}
}  $ & $  {\tiny
\begin{array}{c}2\cdot 3^{1}+O(3^{2})\end{array}
}  $ \\
$  {\tiny
\begin{array}{c}6\end{array}
}  $ & $  {\tiny
\begin{array}{c}-2^{5}\cdot 5^{3}\cdot 13\cdot 1801\end{array}
}  $ & $  {\tiny
\begin{array}{c}\frac{2\cdot 5^{2}}{3}\end{array}
}  $ & $  {\tiny
\begin{array}{c}1\end{array}
}  $ & $  {\tiny
\begin{array}{c}2^{4}\cdot 3^{10}\cdot 5^{2}\end{array}
}  $ & $  {\tiny
\begin{array}{c}1+O(3)\end{array}
}  $ & $  {\tiny
\begin{array}{c}1+O(3)\end{array}
}  $ \\
\end{supertabular}

\end{center}

\begin{center}
\tablefirsthead{%

 \hline

\multicolumn{7}{|c|}{Table II:  form $f$ of  conductor 7 and weight 4.}\\
\multicolumn{7}{|c|}{$L_3^*(f,\sigma,1)=49, L_3^*(f,\sigma,2)=L_3^*(f,\rho,2)=0$.}\\

 \hline  ${\tiny
\begin{array}{c}m\end{array}
}$ & ${\tiny
\begin{array}{c}L_3^*(f,\rho,n)\end{array}
}  $&${\tiny
\begin{array}{c}\mathcal{P}_3(f,\rho,n)\end{array} }$& ${\tiny
\begin{array}{c}\mathcal{P}_3(f,\sigma,n)\end{array}
}$ &${\tiny
\begin{array}{c}N(f,\rho)\end{array}
}
  $&${\tiny
\begin{array}{c} \mathcal{L}_3(f,\sigma,n)\end{array}
}$&${\tiny
\begin{array}{c} \mathcal{L}_3(f,\rho,n)\end{array}
}$ \\\hline }

 \tablehead{\hline \multicolumn{7}{|c|}{Table II:   form $f$ of conductor 7 and weight 4.}\\

\hline  ${\tiny
\begin{array}{c}m\end{array}
}$ & ${\tiny
\begin{array}{c}L_3^*(f,\rho,n)\end{array}
}  $&${\tiny
\begin{array}{c}\mathcal{P}_3(f,\rho,n)\end{array} }$& ${\tiny
\begin{array}{c}\mathcal{P}_3(f,\sigma,n)\end{array}
}$ &${\tiny
\begin{array}{c}N(f,\rho)\end{array}
}
  $&${\tiny
\begin{array}{c} \mathcal{L}_3(f,\sigma,n)\end{array}
}$&${\tiny
\begin{array}{c} \mathcal{L}_3(f,\rho,n)\end{array}
}$ \\\hline }

\tabletail{\hline}

 \tablelasttail{\hline}
\begin{supertabular}{|l@{\hspace{0pt}}|l@{\hspace{0pt}}|l@{\hspace{0pt}}|l@{\hspace{0pt}}|l@{\hspace{0pt}}|l@{\hspace{0pt}}|l@{\hspace{0pt}}|}

$  {\tiny
  \begin{array}{c}2\end{array}
  }  $ & $  {\tiny
  \begin{array}{c}2^{3}\cdot 3\cdot 5\cdot 7^{4}\end{array}
  }  $ & $  {\tiny
  \begin{array}{c}1\end{array}
  }  $ & $  {\tiny
  \begin{array}{c}\frac{5\cdot 7^{2}}{2\cdot 3}\end{array}
  }  $ & $  {\tiny
  \begin{array}{c}2^{4}\cdot 3^{6}\cdot 7^{2}\end{array}
  }  $ & $  {\tiny
  \begin{array}{c}1\cdot 3^{1}+O(3^{2})\end{array}
  }  $ & $  {\tiny
  \begin{array}{c}1\cdot 3^{1}+O(3^{2})\end{array}
  }  $ \\
$  {\tiny
  \begin{array}{c}3\end{array}
  }  $ & $  {\tiny
  \begin{array}{c}2^{2}\cdot 3\cdot 5\cdot 7^{3}\cdot 13^{2}\end{array}
  }  $ & $  {\tiny
  \begin{array}{c}1\end{array}
  }  $ & $  {\tiny
  \begin{array}{c}\frac{2\cdot 7}{3}\end{array}
  }  $ & $  {\tiny
  \begin{array}{c}3^{10}\cdot 7^{2}\end{array}
  }  $ & $  {\tiny
  \begin{array}{c}2\cdot 3^{1}+O(3^{2})\end{array}
  }  $ & $  {\tiny
  \begin{array}{c}2\cdot 3^{1}+O(3^{2})\end{array}
  }  $ \\
$  {\tiny
  \begin{array}{c}5\end{array}
  }  $ & $  {\tiny
  \begin{array}{c}2^{7}\cdot 3\cdot 7^{3}\cdot 71\end{array}
  }  $ & $  {\tiny
  \begin{array}{c}1\end{array}
  }  $ & $  {\tiny
  \begin{array}{c}\frac{2^{3}\cdot 7^{2}\cdot 23}{3\cdot 5^{2}}\end{array}
  }  $ & $  {\tiny
  \begin{array}{c}3^{6}\cdot 5^{4}\cdot 7^{2}\end{array}
  }  $ & $  {\tiny
  \begin{array}{c}1\cdot 3^{1}+O(3^{2})\end{array}
  }  $ & $  {\tiny
  \begin{array}{c}1\cdot 3^{1}+O(3^{2})\end{array}
  }  $ \\
\end{supertabular}
\end{center}

\begin{center}
\tablefirsthead{%

 \hline

\multicolumn{7}{|c|}{Table III:  form $f$ of conductor 5 and weight 6.}\\

\multicolumn{7}{|c|}{$L_3^*(f,\sigma,1)=-400,L_3^*(f,\sigma,2)=62/15,L_3^*(f,\sigma,3)=-31/1125.$}\\

\hline  ${\tiny
\begin{array}{c}m\end{array}
}$ & ${\tiny
\begin{array}{c}L_3^*(f,\rho,n)\end{array}
}  $&${\tiny
\begin{array}{c}\mathcal{P}_3(f,\rho,n)\end{array} }$& ${\tiny
\begin{array}{c}\mathcal{P}_3(f,\sigma,n)\end{array}
}$ &${\tiny
\begin{array}{c}N(f,\rho)\end{array}
}
  $&${\tiny
\begin{array}{c} \mathcal{L}_3(f,\sigma,n)\end{array}
}$&${\tiny
\begin{array}{c} \mathcal{L}_3(f,\rho,n)\end{array}
}$ \\\hline }

 \tablehead{\hline \multicolumn{7}{|c|}{Table III:   form $f$ of  conductor 5 and weight 6.}\\

\hline  ${\tiny
\begin{array}{c}m\end{array}
}$ & ${\tiny
\begin{array}{c}L_3^*(f,\rho,n)\end{array}
}  $&${\tiny
\begin{array}{c}\mathcal{P}_3(f,\rho,n)\end{array} }$& ${\tiny
\begin{array}{c}\mathcal{P}_3(f,\sigma,n)\end{array}
}$ &${\tiny
\begin{array}{c}N(f,\rho)\end{array}
}
  $&${\tiny
\begin{array}{c} \mathcal{L}_3(f,\sigma,n)\end{array}
}$&${\tiny
\begin{array}{c} \mathcal{L}_3(f,\rho,n)\end{array}
}$ \\\hline }

\tabletail{\hline}

 \tablelasttail{\hline}
\begin{supertabular}{|l@{\hspace{0pt}}|l@{\hspace{0pt}}|l@{\hspace{0pt}}|l@{\hspace{0pt}}|l@{\hspace{0pt}}|l@{\hspace{0pt}}|l@{\hspace{0pt}}|}
&$n=1$&&&&&\\

$  {\tiny
  \begin{array}{c}2\end{array}
  }  $ & $  {\tiny
  \begin{array}{c}-2^{12}\cdot 5^{2}\cdot 31\cdot 661\end{array}
  }  $ & $  {\tiny
  \begin{array}{c}1\end{array}
  }  $ & $  {\tiny
  \begin{array}{c}\frac{2^{7}\cdot 5\cdot 11}{3}\end{array}
  }  $ & $  {\tiny
  \begin{array}{c}2^{4}\cdot 3^{6}\cdot 5^{2}\end{array}
  }  $ & $  {\tiny
  \begin{array}{c}1+O(3)\end{array}
  }  $ & $  {\tiny
  \begin{array}{c}1+O(3)\end{array}
  }  $ \\
$  {\tiny
  \begin{array}{c}3\end{array}
  }  $ & $  {\tiny
  \begin{array}{c}-2^{10}\cdot 5^{3}\cdot 13\cdot 31\cdot 2953\end{array}
  }  $ & $  {\tiny
  \begin{array}{c}1\end{array}
  }  $ & $  {\tiny
  \begin{array}{c}\frac{2^{3}\cdot 11}{3}\end{array}
  }  $ & $  {\tiny
  \begin{array}{c}3^{10}\cdot 5^{2}\end{array}
  }  $ & $  {\tiny
  \begin{array}{c}2+O(3)\end{array}
  }  $ & $  {\tiny
  \begin{array}{c}2+O(3)\end{array}
  }  $ \\
$  {\tiny
  \begin{array}{c}5\end{array}
  }  $ & $  {\tiny
  \begin{array}{c}2^{8}\cdot 3\cdot 5\cdot 31\cdot 193\cdot 211\end{array}
  }  $ & $  {\tiny
  \begin{array}{c}1\end{array}
  }  $ & $  {\tiny
  \begin{array}{c}-2^{6}\cdot 11\end{array}
  }  $ & $  {\tiny
  \begin{array}{c}3^{6}\cdot 5^{4}\end{array}
  }  $ & $  {\tiny
  \begin{array}{c}2\cdot 3^{1}+O(3^{2})\end{array}
  }  $ & $  {\tiny
  \begin{array}{c}1\cdot 3^{1}+O(3^{2})\end{array}
  }  $ \\
$  {\tiny
  \begin{array}{c}6\end{array}
  }  $ & $  {\tiny
  \begin{array}{c}-2^{11}\cdot 5^{3}\cdot 31\cdot 137\cdot 39323\end{array}
  }  $ & $  {\tiny
  \begin{array}{c}1\end{array}
  }  $ & $  {\tiny
  \begin{array}{c}\frac{2^{7}\cdot 5\cdot 11}{3}\end{array}
  }  $ & $  {\tiny
  \begin{array}{c}2^{4}\cdot 3^{10}\cdot 5^{2}\end{array}
  }  $ & $  {\tiny
  \begin{array}{c}1+O(3)\end{array}
  }  $ & $  {\tiny
  \begin{array}{c}1+O(3)\end{array}
  }  $ \\
\end{supertabular}

\end{center}

\begin{center}
\tablefirsthead{%

 \hline

\multicolumn{7}{|c|}{Table IV:  form f of conductor  121 and weight 4.}\\

\multicolumn{7}{|c|}{$L_3^*(f,\sigma,s)(n=1,2)=(176,0)$ and $L_3(f,\sigma,2)=L_3(f,\rho,2)=0$.}\\

\hline  ${\tiny
\begin{array}{c}m\end{array}
}$ & ${\tiny
\begin{array}{c}L_3^*(f,\rho,n)\end{array}
}  $&${\tiny
\begin{array}{c}\mathcal{P}_3(f,\rho,n)\end{array} }$& ${\tiny
\begin{array}{c}\mathcal{P}_3(f,\sigma,n)\end{array}
}$ &${\tiny
\begin{array}{c}N(f,\rho)\end{array}
}
  $&${\tiny
\begin{array}{c} \mathcal{L}_3(f,\sigma,n)\end{array}
}$&${\tiny
\begin{array}{c} \mathcal{L}_3(f,\rho,n)\end{array}
}$ \\\hline }

 \tablehead{\hline \multicolumn{7}{|c|}{Table IV:   form $f$ of conductor 121 and weight 4 for $n=1$.}\\

\hline  ${\tiny
\begin{array}{c}m\end{array}
}$ & ${\tiny
\begin{array}{c}L_3^*(f,\rho,n)\end{array}
}  $&${\tiny
\begin{array}{c}\mathcal{P}_3(f,\rho,n)\end{array} }$& ${\tiny
\begin{array}{c}\mathcal{P}_3(f,\sigma,n)\end{array}
}$ &${\tiny
\begin{array}{c}N(f,\rho)\end{array}
}
  $&${\tiny
\begin{array}{c} \mathcal{L}_3(f,\sigma,n)\end{array}
}$&${\tiny
\begin{array}{c} \mathcal{L}_3(f,\rho,n)\end{array}
}$ \\\hline }

\tabletail{\hline}

 \tablelasttail{\hline}
\begin{supertabular}{|l@{\hspace{0pt}}|l@{\hspace{0pt}}|l@{\hspace{0pt}}|l@{\hspace{0pt}}|l@{\hspace{0pt}}|l@{\hspace{0pt}}|l@{\hspace{0pt}}|}

$  {\tiny
  \begin{array}{c}2\end{array}
  }  $ & $  {\tiny
  \begin{array}{c}2^{5}\cdot 3\cdot 11\cdot 17\cdot 37\end{array}
  }  $ & $  {\tiny
  \begin{array}{c}1\end{array}
  }  $ & $  {\tiny
  \begin{array}{c}2^{2}\cdot 3\end{array}
  }  $ & $  {\tiny
  \begin{array}{c}2^{4}\cdot 3^{6}\cdot 11^{4}\end{array}
  }  $ & $  {\tiny
  \begin{array}{c}2\cdot  3^{2}+O(3^{3})\end{array}
  }  $ & $  {\tiny
  \begin{array}{c}1\cdot  3^{1}+O(3^{2})\end{array}
  }  $ \\
$  {\tiny
  \begin{array}{c}3\end{array}
  }  $ & $  {\tiny
  \begin{array}{c}2^{5}\cdot 5\cdot 11\cdot 4373\end{array}
  }  $ & $  {\tiny
  \begin{array}{c}1\end{array}
  }  $ & $  {\tiny
  \begin{array}{c}\frac{2^{2}}{3}\end{array}
  }  $ & $  {\tiny
  \begin{array}{c}3^{10}\cdot 11^{4}\end{array}
  }  $ & $  {\tiny
  \begin{array}{c}2+O(3)\end{array}
  }  $ & $  {\tiny
  \begin{array}{c}2+O(3)\end{array}
  }  $ \\
$  {\tiny
  \begin{array}{c}5\end{array}
  }  $ & $  {\tiny
  \begin{array}{c}2^{5}\cdot 3^{3}\cdot 11\cdot 2069\end{array}
  }  $ & $  {\tiny
  \begin{array}{c}1\end{array}
  }  $ & $  {\tiny
  \begin{array}{c}\frac{2^{8}\cdot 3}{5^{2}}\end{array}
  }  $ & $  {\tiny
  \begin{array}{c}3^{6}\cdot 5^{4}\cdot 11^{4}\end{array}
  }  $ & $  {\tiny
  \begin{array}{c}2\cdot 3^{2}+O(3^{3})\end{array}
  }  $ & $  {\tiny
  \begin{array}{c}1\cdot 3^{3}+O(3^{4})\end{array}
  }  $ \\
$  {\tiny
  \begin{array}{c}6\end{array}
  }  $ & $  {\tiny
  \begin{array}{c}2^{3}\cdot 3^{2}\cdot 11\cdot 83\cdot 2297\end{array}
  }  $ & $  {\tiny
  \begin{array}{c}1\end{array}
  }  $ & $  {\tiny
  \begin{array}{c}2^{2}\cdot 3\end{array}
  }  $ & $  {\tiny
  \begin{array}{c}2^{4}\cdot 3^{10}\cdot 11^{4}\end{array}
  }  $ & $  {\tiny
  \begin{array}{c}2\cdot 3^{2}+O(3^{3})\end{array}
  }  $ & $  {\tiny
  \begin{array}{c}2\cdot 3^{2}+O(3^{3})\end{array}
  }  $ \\
\end{supertabular}
\end{center}

\bigskip

In the remaining four tables, we give some intriguing integrality
and squareness assertions for the L-values computed in the previous
four tables. Although we do not enter into any detailed discussion
in the present paper, it seems highly likely that these phenomena
can be explained via the Bloch-Kato conjecture, and Flach's
motivic generalization of the Cassels-Tate pairing. We define
$M$ to be the product of the distinct primes dividing $m$, but excluding the prime 3.
Let $N$ denote the conductor of our primitive form $f$. For $n=1,\dots, k-1$, we define
\begin{equation}\label{An}
A_n(f) = |L_3^*(f,\rho,n)|M^n\epsilon _{3}(\rho)^{(n-1)}/4
\end{equation}

\bigskip

In Table V, for the form $f$ of conductor 5 and weight 4, define

\begin{align*}B_1(f)&=A_1(f)/(2^2 \times 5^3 \times 13),\\
 B_2(f) &=A_2(f)/(5^2 \times 13).\end{align*}

\begin{center}
\tablefirsthead{

 \hline

\multicolumn{3}{|c|}{Table V:  form $f$ of conductor 5 and weight 4.}\\

 \hline  $m$&$B_1(f)$&$\sqrt{B_2(f)}$ \\\hline }

 \tablehead{\hline \multicolumn{3}{|c|}{Table V:  form $f$ of conductor 5 and weight 4.}\\

\hline $m$&$B_1(f)$&$\sqrt{B_2(f)}$  \\\hline }

\tabletail{\hline}

 \tablelasttail{\hline}
\begin{supertabular}{|l@{\hspace{20pt}}|l@{\hspace{0pt}}|l@{\hspace{0pt}}|}
$2 $&$ 2^{2}\cdot 7 $&$ 2$\\
$3 $&$ 5\cdot 41 $&$ 1$\\
$6 $&$ 2^{2}\cdot 1801 $&$ 2\cdot 7$\\
$7 $&$ 2^{4}\cdot 23\cdot 41 $&$ 2^{2}$\\
$11 $&$ 2^{6}\cdot 2311 $&$ 2^{2}\cdot 11$\\
$12 $&$ 2^{3}\cdot 839 $&$ 2^{3}$\\
$13 $&$ 11\cdot 13\cdot 43\cdot 53 $&$ 1$\\
$14 $&$ 2^{2}\cdot 5\cdot 7\cdot 13\cdot 251 $&$ 2\cdot 5$\\
$17 $&$ 31\cdot 167 $&$ 5$\\
$19 $&$ 5\cdot 43^{2} $&$ 13$\\
$21 $&$ 3425341 $&$ 67$\\
$22 $&$ 2^{3}\cdot 43\cdot 13841 $&$ 2^{2}\cdot 11$\\
$23 $&$ 2^{4}\cdot 3^{5}\cdot 1409 $&$ 2^{2}\cdot 3\cdot 5$\\
$26 $&$ 2^{2}\cdot 13\cdot 887 $&$ 2\cdot 13$\\
$28 $&$ 2^{2}\cdot 503 $&$ 2$\\
$29 $&$ 11\cdot 1678031 $&$ 109$\\
$31 $&$ 5\cdot 79\cdot 62351 $&$ 151$\\
$33 $&$ 5\cdot 11^{2}\cdot 19\cdot 2879 $&$ 5\cdot 7$\\
$34 $&$ 2^{4}\cdot 17\cdot 142427 $&$ 2^{2}$\\
$37 $&$ 2^{4}\cdot 3^{2}\cdot 5\cdot 367 $&$ 2^{2}\cdot 3$\\
$39 $&$ 2^{6}\cdot 71\cdot 17489 $&$ 2^{2}\cdot 11$\\
$41 $&$ 17\cdot 31\cdot 211\cdot 941 $&$ 11$\\
$42 $&$ 2^{2}\cdot 19\cdot 859\cdot 1801 $&$ 2\cdot 149$\\
$43 $&$ 2^{2}\cdot 7\cdot 19\cdot 251\cdot 491 $&$ 2\cdot 5^{2}$\\
$44 $&$ 2^{2}\cdot 11\cdot 421 $&$ 2\cdot 7$\\
$46 $&$ 2^{2}\cdot 3^{3}\cdot 7283 $&$ 2\cdot 3$\\
$47$&$2^4\cdot 23^2\cdot 22567$&$2^4\cdot 13$\\
$51 $&$ 2^{4}\cdot 5\cdot 13\cdot 278591 $&$ 2^{2}\cdot 101$\\
$52 $&$ 2^{2}\cdot 2513617 $&$ 2$\\
$53 $&$ 5\cdot 290161 $&$ 29$\\
$57 $&$ 2^{2}\cdot 61\cdot 503\cdot 4241 $&$ 2\cdot 7^{2}$\\
$58 $&$ 2^{6}\cdot 9208039 $&$ 2^{6}\cdot 5$\\
$59 $&$ 2^{4}\cdot 5\cdot 23\cdot 397\cdot 853 $&$ 2^{2}\cdot 11^{2}$\\
$62 $&$ 2^{2}\cdot 307\cdot 2879 $&$ 2\cdot 5\cdot 7$\\
$66 $&$ 2^{5}\cdot 5\cdot 6458773 $&$ 2^{3}\cdot 11$\\
$67 $&$ 2^{2}\cdot 13^{2}\cdot 19^{2}\cdot 4759 $&$ 2\cdot 103$\\
$68 $&$ 2^{2}\cdot 10484557 $&$ 2\cdot 7\cdot 13$\\
$69 $&$ 2^{2}\cdot 3^{3}\cdot 857\cdot 15733 $&$ 2\cdot 3\cdot 31$\\
$71 $&$ 2^{2}\cdot 7\cdot 31\cdot 79\cdot 101 $&$ 2\cdot 29$\\
$73 $&$ 5\cdot 17\cdot 47\cdot 1831 $&$ 43$\\
$74 $&$ 2^{2}\cdot 3^{4}\cdot 11\cdot 523\cdot 1031 $&$ 2\cdot 3^{3}$\\
$76 $&$ 2^{5}\cdot 311\cdot 7297 $&$ 2^{3}\cdot 13$\\
$77 $&$ 2^{4}\cdot 7\cdot 11\cdot 2377\cdot 60913 $&$ 2^{2}\cdot 101$\\
$82 $&$ 2^{3}\cdot 5\cdot 73\cdot 4817 $&$ 2^{2}\cdot 11$\\
$83 $&$ \text{?}$&$ 2^{2}\cdot 3^{2}\cdot 7$\\
$84 $&$ 2^{2}\cdot 7\cdot 431\cdot 10259 $&$ 2\cdot 199$\\
$89 $&$ 2^{2}\cdot 3^{5}\cdot 71\cdot 293 $&$ 2\cdot 3^{2}\cdot 5$\\
$91 $&$ 5\cdot 4519393 $&$ 1$\\
$92 $&$ 2^{2}\cdot 3^{2}\cdot 157\cdot 31019 $&$ 2\cdot 3^{2}$\\
$93 $&$ \text{?} $&$ 2\cdot 31^{3}\cdot 179$\\
$94 $&$ \text{?}$&$ 2\cdot 11\cdot 19$\\
$97 $&$ \text{?} $&$ 2^{2}\cdot 3^{2}\cdot 5^{2}$\\

\end{supertabular}
\end{center}

\bigskip

In Table VI, for the form $f$ of conductor 7 and weight 4, define

$$
B_1(f) = A_1(f)/(7^3 \times 5).
$$

\begin{center}
\tablefirsthead{%

 \hline

\multicolumn{2}{|c|}{Table VI: form $f$ of conductor 7 and weight 4. }\\

 \hline  $m$&$B_1(f)$ \\\hline }

 \tablehead{\hline \multicolumn{2}{|c|}{Table VI: form $f$ of conductor 7 and weight 4. }\\

\hline $m$&$B_1(f)$ \\\hline }

\tabletail{\hline}

 \tablelasttail{\hline}
\begin{supertabular}{|l@{\hspace{20pt}}|l@{\hspace{20pt}}|}

$2$ & $2^{2}\cdot 3\cdot 7$ \\
$3$ & $3\cdot 13^{2}$ \\
$5$ & $2^{5}\cdot 3\cdot 71$ \\
$6$ & $2^{3}\cdot 3\cdot 7\cdot 113$ \\
$7$ & $3\cdot 223$ \\
$10$ & $2^{2}\cdot 239$ \\
$11$ & $3\cdot 211\cdot 499$ \\
$12$ & $2^{2}\cdot 3\cdot 7\cdot 241$ \\
$13$ & $2^{6}\cdot 3\cdot 5\cdot 773$ \\
$14$ & $2^{2}\cdot 3\cdot 41\cdot 59$ \\
$15$ & $3\cdot 5\cdot 13\cdot 43\cdot 179$ \\
$17$ & $3^{2}\cdot 1223$ \\
$19$ & $2^{9}\cdot 37$ \\
$20$ & $2^{6}\cdot 3\cdot 1213$ \\
$21$ & $2^{11}\cdot 3\cdot 29$ \\
$22$ & $2^{3}\cdot 3\cdot 19\cdot 28277$ \\
$23$ & $3^{3}\cdot 47\cdot 10463$ \\
$26$ & $2^{2}\cdot 7\cdot 3917$ \\
$28$ & $2^{3}\cdot 13$ \\
$29$ & $3\cdot 7\cdot 1904647$ \\
$30$ & $2^{2}\cdot 3\cdot 19\cdot 266839$ \\
$31$ & $2^{6}\cdot 3\cdot 307267$ \\
$33$ & $2^{5}\cdot 3\cdot 849221$ \\
$34$ & $2^{9}\cdot 3^{3}\cdot 83\cdot 101$ \\
$35$ & $43\cdot 191$ \\
$37$ & $2^{5}\cdot 15937$ \\
$38$ & $2^{4}\cdot 3\cdot 5\cdot 864947$ \\
$39$ & $2^{6}\cdot 3\cdot 957811$ \\
$41$ & $2^{2}\cdot 3\cdot 5\cdot 13\cdot 173\cdot 1693$ \\
$42$ & $2^{2}\cdot 3\cdot 37\cdot 15601$ \\

\end{supertabular}
\end{center}

In Table VII, for  the form $f$ of conductor 5 and weight 6, define
\begin{align*}
B_1(f)&=A_1(f)/(2^{6}\cdot 31\cdot 5^2),\\
B_2(f)&=A_2(f)/(2^4\cdot 31),\\
B_3(f)&=A_3(f)\times 5 /(2^4\cdot 31).
\end{align*}

\begin{center}
\tablefirsthead{%

 \hline

\multicolumn{4}{|c|}{Table VII: form $f$ of conductor 5 and weight 6. }\\

 \hline  $m$&$B_1(f)$&$B_2(f)$&$\sqrt{B_3(f)}$\\\hline }

 \tablehead{\hline \multicolumn{4}{|c|}{Table VII: form $f$ of conductor 5 and weight 6. }\\

\hline $m$&$B_1(f)$&$B_2(f)$&$\sqrt{B_3(f)}$ \\\hline }

\tabletail{\hline}

 \tablelasttail{\hline}
\begin{supertabular}{|l@{\hspace{20pt}}|l@{\hspace{20pt}}|l@{\hspace{20pt}}|l@{\hspace{20pt}}|}
$2$ & $2^{5}\cdot 661$&$1759$ &$1$\\
$3$ & $2^{2}\cdot 5\cdot 13\cdot 2953$&$2^{2}\cdot 5\cdot 1223$ &$2$\\
$5$ & $3\cdot 193\cdot 211$&$3\cdot 5^{2}\cdot 13\cdot 37$ &$0$\\
$6$ & $2^{4}\cdot 5\cdot 137\cdot 39323$&$19\cdot 47\cdot 5531$ &$59$\\
$7$ & $2^{4}\cdot 7\cdot 14230919$&$2^{2}\cdot 47\cdot 53813$ &$47$\\
$10$ & $3\cdot 1097$&$3\cdot 5^{2}\cdot 31$ &$0$\\
$11$ & $2^{5}\cdot 5\cdot 971\cdot 592759$&$2^{3}\cdot 28000571$ &$181$\\
$12$ & $2^{4}\cdot 7^{2}\cdot 533063$&$2^{2}\cdot 7\cdot 145543$ &$2^{3}$\\
$13$ & $2^{2}\cdot 7\cdot 11\cdot 211\cdot 6591061$&$2^{3}\cdot 112051757$ &$13\cdot 31$\\
$14$ & $2^{5}\cdot 5^{2}\cdot 1082124649$&$5^{2}\cdot 65780839$ &$5^{2}\cdot 19$\\
$15$ & $2^{2}\cdot 3\cdot 13697\cdot 15101$&$2^{3}\cdot 3\cdot 5^{2}\cdot 103\cdot 1559$ &$0$\\
$17$ & $2^{2}\cdot 5^{3}\cdot 7\cdot 65777$&$2^{2}\cdot 491\cdot 971$ &$5$\\
$18$ & $2^{4}\cdot 7^{2}\cdot 533063$&$2^{2}\cdot 7\cdot 145543$ &$2^{3}$\\
$19$ & $2^{2}\cdot 11\cdot 14243891$&$2\cdot 5\cdot 418273$ &$2^{2}$\\
$20$ & $3\cdot 59\cdot 387077$&$3\cdot 5^{2}\cdot 96457$ &$0$\\
$21$ & $2^{2}\cdot 29\cdot 104789\cdot 2583353$&$2^{3}\cdot 19^{2}\cdot 647\cdot 11827$ &$2\cdot 191$\\
$23$ & $2^{4}\cdot 3^{3}\cdot 5\cdot 32517200203$&$2^{3}\cdot 3^{3}\cdot 1117\cdot 156733$ &$3\cdot 5$\\
$26$ & $2^{4}\cdot 699507967$&$7\cdot 6916561$ &$47$\\
$28$ & $2^{6}\cdot 5\cdot 29\cdot 41\cdot 113$&$17\cdot 39383$ &$19$\\
$29$ & $2^{2}\cdot 5\cdot 19\cdot 37\cdot 41633381443$&$2^{2}\cdot 283\cdot 3323\cdot 64067$ &$2\cdot 757$\\
$30$ & $3\cdot 13^{2}\cdot 17\cdot 53\cdot 1051\cdot 2713$&$3\cdot 5^{2}\cdot 7^{2}\cdot 3320281$ &$0$\\
$31$ & $2^{2}\cdot 5^{2}\cdot 1597\cdot 25447\cdot 254627$&$2^{2}\cdot 5\cdot 73\cdot 219638621$ &$2\cdot 967$\\
$33$ & $2^{2}\cdot 79\cdot 5727093605801$&$2^{2}\cdot 5^{2}\cdot 43\cdot 109868293$ &$5\cdot 31$\\
$34$ & $2^{5}\cdot 7\cdot 17\cdot 23\cdot 227\cdot 130914857$&$2^{4}\cdot 487\cdot 122916679$ &$2\cdot 1867$\\
$35$ & $2^{8}\cdot 3\cdot 7\cdot 46049$&$2^{4}\cdot 3\cdot 5^{2}\cdot 10729$ &$0$\\
$37$ & $2^{4}\cdot 3^{2}\cdot 5\cdot 181\cdot 199\cdot 9743$&$2^{4}\cdot 3^{2}\cdot 5\cdot 13\cdot 59\cdot 829$ &$3\cdot 7^{2}$\\
$38$ & $2^{7}\cdot 5^{2}\cdot 7\cdot 61\cdot 27077\cdot 185057$&$1657646829583$ &$31\cdot 149$\\
$39$ & $2^{5}\cdot 2957\cdot 86182236263$&$2^{3}\cdot 1039\cdot 3011\cdot 62311$ &$7\cdot 97$\\
$41$ & $2^{2}\cdot 3303519970879679$&$2\cdot 53\cdot 709\cdot 36628831$ &$881$\\
$42$ & $2^{5}\cdot 7^{3}\cdot 267139\cdot 5797783$&$59\cdot 59147190533$ &$4919$\\
$43$ & $2^{5}\cdot 19\cdot 638839\cdot 52230109$&$2^{3}\cdot 482148655367$ &$5\cdot 13\cdot 67$\\

\end{supertabular}
\end{center}

In Table VIII, for the CM form $f$ of conductor  121 and weight 4, define

$$
B_1(f)=A_1(f)/(2^2\cdot 11).$$

\begin{center}
\tablefirsthead{%

 \hline

\multicolumn{2}{|c|}{Table VIII: form $f$ of conductor 121 and weight 4. }\\

 \hline  $m$&$B_1(f)$ \\\hline }

 \tablehead{\hline \multicolumn{2}{|c|}{Table VIII: form $f$ of conductor 121 and weight 4.}\\

\hline $m$&$B_1(f)$ \\\hline }

\tabletail{\hline}

 \tablelasttail{\hline}
\begin{supertabular}{|l@{\hspace{20pt}}|l@{\hspace{20pt}}|}
$2$ & $2^{2}\cdot 3\cdot 17\cdot 37$ \\
$3$ & $2\cdot 5\cdot 4373$ \\
$5$ & $2\cdot 3^{3}\cdot 5\cdot 2069$ \\
$6$ & $3^{2}\cdot 83\cdot 2297$ \\
$7$ & $2\cdot 5\cdot 349\cdot 863$ \\
$10$ & $3^{2}\cdot 5\cdot 13\cdot 211$ \\
$11$ & $2^{4}\cdot 11^{2}$ \\
$12$ & $2^{2}\cdot 3\cdot 13\cdot 31\cdot 367$ \\
$14$ & $2^{4}\cdot 3\cdot 5\cdot 439\cdot 1129$ \\
$17$ & $2\cdot 3\cdot 29\cdot 8219$ \\
$19$ & $2\cdot 5^{3}\cdot 11\cdot 29\cdot 31$ \\
$20$ & $2^{4}\cdot 3^{2}\cdot 156241$ \\

\end{supertabular}
\end{center}


\newpage

\begin{thebibliography}{99}


\bibitem{BD1}{\sc T. Bouganis, V. Dokchitser}, \newblock {\it Algebraicity of  $L$-values for elliptic curves in a false Tate curve tower}, Proc. Camb. Phil. Soc. {\bf 142} (2007), 193--204.

\bibitem{CFKSV}{\sc J. Coates, T. Fukaya, K. Kato, R. Sujatha, O. Venjakob}, \newblock{\it
The GL$_2$ main conjecture for elliptic curves without complex multiplication}, \newblock{Publ. Math. Inst. Hautes \'Etudes Sci.} {\bf 101} (2005).

\bibitem{DT}{\sc D. Delbourgo, T. Ward}, \newblock{\it The growth of CM periods over false Tate extensions},
Experiment. Math. {\bf 19} (2010), 195-210.


\bibitem{TD}{\sc T. Dokchitser}, \newblock{\it Computing Special Values of Motivic $L$-functions},
Experiment. Math. {\bf 13} (2004), 137-150.



\bibitem{DD}{\sc T. Dokchitser, V. Dokchitser}, \newblock {\it Computations in non-commutative
Iwasawa theory}, Proc. London Math. Soc. {\bf 94} (2006), 211--272.

\bibitem{VD}{\sc V. Dokchitser}, \newblock {\it  Root numbers of non-abelian twists of elliptic
curves (with appendix by Tom Fisher)}, Proc. London Math. Soc. {\bf 3} 91 (2005), 300--324.

\bibitem{DSW}{\sc N.~Dummigan, W.\thinspace{}A. Stein, and M.~Watkins}, \newblock {\it Constructing
  Elements in Shafarevich-Tate Groups of Modular Motives}, Number theory and
  algebraic geometry, ed. by Miles Reid and Alexei Skorobogatov \textbf{303}
  (2003), 91--118.

\bibitem{FK}{\sc T. Fukaya, K.Kato}, \newblock{\it A formulation of conjectures on $p$-adic zeta functions in noncommutative Iwasawa theory}, \newblock Proceedings of the St. Petersburg Mathematical Society. Vol. XII, 1�85, Amer. Math. Soc. Transl. Ser. 2, {\bf 219},\newblock {Amer. Math. Soc., Providence, RI, 2006.}



\bibitem{G}{\sc R. Greenberg}, \newblock{\it Iwasawa theory for $p$-adic representations}, \newblock{Algebraic Number Theory, Advanced Studies in Pure Mathematics}, \newblock{Academic Press}, {\bf 17} (1989), 97-137.

\bibitem{H-M}{\sc Y. Hachimori, K. Matsuno}, \newblock{\it An analogue of Kida's formula for the Selmer groups of elliptic curves}, J. Algebraic Geom. {\bf 8}  (1999), 581--601.


\bibitem{F} {\sc B. Hart et.al.} \newblock{\it FLINT: Fast library for number theory}, \url{http://www.flintlib.org}

\bibitem{Ka}{\sc M. Kakde}, \newblock{\it The main conjecture of Iwasawa theory for totally real fields},  arXiv:1008.0142.

\bibitem{K}{\sc K. Kato}, \newblock {\it $ p$-adic Hodge theory and values of zeta functions of modular forms}, in Cohomologies $p$-adiques et applications arithm\'etiques III,  Ast\'erisque  {\bf 295}  (2004),  117--290.

\bibitem{K1}{\sc K. Kato}, \newblock{\it $K_1$ of some non-commutative completed group rings},
\newblock{K-theory} {\bf 34} (2005), 99--140.



\bibitem{DK}{\sc D. Kim}, \newblock{\it $p$-adic $L$-functions over false Tate extensions}, preprint.

\bibitem{M}{\sc Y. Manin}, \newblock{Values of $p$-adic Hecke series at lattice points of the critical strip}, \newblock(Russian) Mat. Sb. (N.S.) {\bf 93} (135) (1974), 621Ð626, 631.

\bibitem{MW}{\sc B. Mazur, A. Wiles},  {\it On $p$-adic analytic families of Galois representations}, Compositio Math.  {\bf 59} (1986),  231--264.


\bibitem{MI}{\sc T. Miyake}, \newblock{\it Modular Forms},  Springer (1989).


\bibitem{P} \newblock{\sc PARI}, \newblock { A computer algebra system designed for fast computations in number theory}, \url{ http://pari.math.u-bordeaux.fr/}.


\bibitem{PW}{\sc R. Pollack, T. Weston}, \newblock{\it Kida's formula and congruences},
Documenta Mathematica, Special Volume (2006), 615-630.


\bibitem{RW}{\sc J. Ritter, A. Weiss}, \newblock{\it On the `main conjecture'  of equivariant Iwasawa theory},  J. Amer. Math. Soc. {\bf 24}(2011), 1015-1050.

\bibitem{R}{\sc D. Rohrlich}, \newblock{\it L-functions and division towers},
Math. Ann.  {\bf 281} (1988), 611--632.

\bibitem[SAGE]{sage}
W.\thinspace{}A. Stein et~al., \emph{{S}age {M}athematics {S}oftware ({V}ersion
  4.7.2)}, The Sage Development Team, 2011, \url{http://www.sagemath.org}.

\bibitem{S1}{\sc G. Shimura},  \newblock  {\it On the special values of the zeta functions
associated with cusp forms}, Comm. Pure Appl. Math. {\bf 29} (1976), 783--804.

\bibitem{S2}{\sc G. Shimura}, \newblock
{\it On the periods of modular forms}, Math. Ann. {\bf 229}  (1977), 211--221.

\bibitem{S3}{\sc G. Shimura}, \newblock {\it The special values of zeta functions associated with Hilbert modular forms}, Duke Math. Jour. {\bf 45} (1978), 637--679.

\bibitem{TatN}{\sc J. Tate}, \newblock{\it Number theoretic background},
in: Automorphic forms, representations and L-functions, Part 2
(ed. A. Borel and W. Casselman), Proc. Symp. in Pure Math.
Math. Ann.  {\bf 281} (1988), 611--632.



\end{thebibliography}
\end{document}